\documentclass[11pt]{amsart}
\usepackage{amssymb,amsmath,amsthm}
\usepackage{verbatim}
\usepackage{enumerate}
\usepackage{graphicx}
\usepackage{epsfig}
\usepackage{color, soul}
\usepackage[all]{xy}
\usepackage{wasysym}
\usepackage{url}
\usepackage{mathrsfs}
\usepackage{cancel}

\DeclareMathOperator{\cb}{cb}

\DeclareMathOperator{\CB}{\mathcal{CB}}
\DeclareMathOperator{\tr}{tr}
\DeclareMathOperator{\proj}{proj}

\DeclareMathOperator{\ex}{ex}
\DeclareMathOperator{\amconv}{amconv}

\DeclareMathOperator{\fin}{fin}

\newcommand{\n}[1]{ \left\|#1\right\| }

\newcommand{\N}{{\mathbb{N}}}

\newcommand{\C}{{\mathbb{C}}}

\newcommand{\pair}[2]{{\langle #1, #2 \rangle}}
\newcommand{\mpair}[2]{{\langle\langle #1, #2 \rangle\rangle}}

\newcommand\restr[2]{{
  \left.\kern-\nulldelimiterspace 
  #1 
  \vphantom{\big|} 
  \right|_{#2} 
  }}

\newtheorem{theorem}{Theorem}[section]
\newtheorem{lemma}[theorem]{Lemma}
\newtheorem{definition}[theorem]{Definition}
\newtheorem{corollary}[theorem]{Corollary}
\newtheorem{proposition}[theorem]{Proposition}
\newtheorem{remark}[theorem]{Remark}
\newtheorem{example}[theorem]{Example}

\numberwithin{equation}{section}

\def\ker{{\rm ker\, }}

\begin{document}

\title{Revisiting Operator $p$-Compact Mappings}

\author[J.A. Ch\'avez-Dom\'inguez]{Javier Alejandro Ch\'avez-Dom\'inguez}
\address{Department of Mathematics, University of Oklahoma, Norman, OK 73019-3103,
USA} \email{jachavezd@ou.edu}

\author[V. Dimant]{Ver\'onica Dimant}
\address{Departamento de Matem\'{a}tica y Ciencias, Universidad de San
Andr\'{e}s, Vito Dumas 284, (B1644BID) Victoria, Buenos Aires,
Argentina and CONICET} \email{vero@udesa.edu.ar}

\author[D. Galicer]{Daniel Galicer}
\address{Departamento de Matem\'{a}ticas y Estad\'{\i}stica, Universidad T. Di Tella, Av. Figueroa Alcorta 7350 (1428), Buenos Aires, Argentina and CONICET.
On leave from Departamento de Matem\'{a}tica, Facultad de Ciencias Exactas y Naturales, Universidad de Buenos Aires, (1428) Buenos Aires,
Argentina} \email{daniel.galicer@utdt.edu}
\date{}

\thanks{The first-named author was partially supported by NSF grants DMS-1900985 and DMS-2247374. The second-named author was partially supported by CONICET PIP 11220200101609CO and ANPCyT PICT 2018-04104. The third-named author was partially supported by CONICET-PIP 11220200102366CO
 and ANPCyT PICT 2018-4250.}

\begin{abstract}
We continue our study of the mapping ideal of operator $p$-compact maps, previously introduced by the authors. Our approach embraces a more geometric perspective, delving into the interplay between operator $p$-compact mappings and matrix sets, specifically we provide a quantitative notion of operator $p$-compactness for the latter.
In particular, we consider operator $p$-compactness in the bidual and its relation with this property in the original space.
Also, we deepen our understanding of the connections between these mapping ideals and other significant ones (e.g., completely $p$-summing, completely $p$-nuclear).
\end{abstract}

\maketitle

\section{Introduction}

To establish the metric theory of tensor products, Grothendieck provided a characterization of compactness in Banach spaces with independent significance \cite[Chap. I, p. 112]{grothendieck1955produits}. According to his formulation, relatively compact sets are precisely those which are contained within the absolutely convex hull of a null sequence. That is, a subset \(K\) of a Banach space \(X\) is relatively compact if and only if there exists a null sequence \((x_n)_{n \in \mathbb{N}}\) in $c_0(X)$ such that

\begin{equation}\label{groth compact definition}
K\subseteq \bigg\{\sum_{n=1}^{\infty} \alpha_n x_n \colon \sum_{n=1}^{\infty} \lvert \alpha_n \rvert \leq 1\bigg\} \quad.
\end{equation}

Inspired by Grothendieck's result Sinha and Karn \cite{Sinha-Karn} introduced a stronger notion of comptactness, called relatively \(p\)-compact sets. These are sets determined in a manner similar to \eqref{groth compact definition}, but restricting to  \(p\)-summable sequences rather than null sequences. More precisely, for \(1 \leq p < \infty\) and \(\frac{1}{p} + \frac{1}{p'} = 1\), a subset \(K \subset X\) is said to be \emph{relatively \(p\)-compact} if there exists a sequence $(x_n)_n \in \ell_p(X)$ such that
\begin{equation}\label{p-compact}
K\subseteq \bigg\{\sum_{n=1}^{\infty} \alpha_n x_n \colon \sum_{n=1}^{\infty} \lvert \alpha_n \rvert^{p'} \leq 1\bigg\}.
\end{equation}

In the limiting case \(p = 1\), the definition is modified as usual. Consequently, classical compact sets can be viewed as ``infinite-compact''. Furthermore, a monotonicity relation holds: if \(1 \leq q \leq p \leq \infty\), any relatively \(q\)-compact set is also relatively \(p\)-compact. Thus, \(p\)-compactness reveals more intricate structures on compact sets.
In analogy with compact linear maps, Sinha and Karn \cite{Sinha-Karn} also defined $p$-compact maps as those which map the closed unit ball of the domain into a relatively $p$-compact subset of the codomain.
Since their introduction there has been great interest in this class of maps, together with some other closely related notions, from a number of different perspectives including
the theories of operator ideals and tensor norms,   various associated approximation properties, structural properties of sets and sequences and  infinite dimensional complex analysis (see, e. g. \cite{MR2651671,Choi-Kim,MR2967309, Galicer-Lassalle-Turco, MR3091821, Pietsch-p-compact, MR3394623, Aron-Caliskan-Garcia-Maestre, MR3720927} and the references therein).

In previous work \cite{ChaDiGa-Operator-p-compact} we defined corresponding notions for operator spaces, both operator $p$-compact mappings and operator $p$-compact (matrix) sets. We mostly approached the issue from the point of view of tensor norms and operator ideals, in the spirit of \cite{Galicer-Lassalle-Turco}, studying mainly the mappings and only briefly touching on the sets. In this companion work we continue our study with a greater emphasis on the sets, based on a geometric interpretation of \eqref{p-compact}.
Specifically, for $x=(x_n)_{n \in \mathbb{N}} \in \ell_p(X)$ denote by $\Theta^x : \ell_{p'} \to X$ the multiplication mapping given by $\alpha \mapsto \sum_{n=1}^{\infty} \alpha_n x_n$, which is well-defined thanks to H\"older's inequality. Note  that then $K \subset X$ is relatively $p-$compact if and only if $K$ is included in the image $\Theta^x(B_{\ell_{p'}})$ for some $x=(x_n)_{n \in \mathbb{N}} \in \ell_p(X)$ (compare with \eqref{p-compact}).
To state the corresponding definition from \cite{ChaDiGa-Operator-p-compact} in the context of operator spaces, first recall that a \emph{matrix set $\mathbf{K} = (K_n)_n$ over an operator space $V$} is a sequence of subsets $K_n \subseteq M_n(V)$ for each $n\in\N$. A typical example of a matrix set over $V$ is the \emph{closed matrix unit ball of $V$} given by
$\mathbf{B}_V = \big( B_{M_n(V)} \big)_n$. For a linear map $T : V \to W$ between operator spaces, the expression $T(\mathbf{K})$ denotes the matrix set $\big(T_n(K_n)\big)_n$ where $T_n$ is the $n$-th amplification of $T$. For two matrix sets $\mathbf{K} = (K_n)_n$ and $\mathbf{L} = (L_n)_n$ defined over the same operator space $V$, we denote $\mathbf{K} \subseteq \mathbf{L}$ to signify that $K_n \subseteq L_n$ holds for all $n \in \mathbb{N}$.

Note that the language of matrix sets allow us to more transparently see the analogy between bounded and completely bounded linear maps. A linear map $T : X \to Y$ between Banach spaces is bounded with norm at most $C$ if and only if $T(B_X) \subseteq C B_Y$, whereas a linear map $T : V \to W$ between operator spaces is completely bounded with completely bounded norm at most $C$ if and only if $T(\mathbf{B}_V) \subseteq C \mathbf{B}_W$.
In operator space theory the Schatten $p$-class $\mathcal{S}_p$ often plays the role of a noncommutative version of the space $\ell_p$, see for example Pisier's notion of a completely $p$-summing map \cite{Pisier-Asterisque}, which provides the final ingredient needed for an operator space version of $p$-compactness: we say that a matrix set $\mathbf{K}$ over an operator space $V$ is operator $p$-compact if there exists $v=(v_{ij})_{i,j=1}^\infty$ in the $V$-valued Schatten space $\mathcal{S}_p[V]$ such that $\mathbf{K}\subset \Theta^v(\mathbf{B}_{\mathcal{S}_p'})$, where $\Theta^v: \mathcal{S}_p' \to V$ is the mapping given by $(\alpha_{ij}) \mapsto \sum_{i,j=1}^{\infty} \alpha_{ij} v_{ij}$ (see Section \ref{sec-sets-and-mappings} for the technical details).
The most significant difference between our previous approach \cite{ChaDiGa-Operator-p-compact} and the present paper is the introduction of a quantitative measure of $p$-compactness for the relatively operator $p$-compact matrix sets analogous to the classical one (Definition \ref{def-measure-of-p-compactness}). 

We will now describe the contents of the rest of the paper. Section \ref{sec-preliminaries} introduces notation and preliminaries.
In Section \ref{sec-sets-and-mappings}, we make precise the definitions referenced in the previous paragraph, and show some basic properties of operator $p$-compact matrix sets and mappings, mostly mirroring the Banach space case. These include the aforementioned monotonicity (Proposition \ref{prop-monotonicity-compactness}) and a factorization theorem (Theorem \ref{thm-factorization-Choi-Kim}).
We recall some elementary examples of operator $p$-compact mappings, and compare our notion with alternative concepts of compactness in the operator space setting \cite{Webster,Yew}.
In particular, we show that this circle of ideas provides an answer to a question of Webster which we previously emphasized in \cite[Problem 4.3]{ChaDiGa-Operator-p-compact}.

In Section \ref{sec-adjoints}, we study the adjoints of operator $p$-compact mappings. Just as in the Banach space case, there is a close connection with the right $p$-nuclear maps. However, the results are not completely analogous because of the usual complications when moving to the noncommutative setting (for example, the lack of local reflexivity). 
Nevertheless, we are able to calculate the operator $p$-compact norm of other mappings beyond the aforementioned elementary examples (Corollary \ref{cor-examples-p-compact}) and also show that the monotonicity relation for operator $p$-compactness is strict (Corollary \ref{cor: p-not-q}).

Next, in Section \ref{sect-weakly-p-compact} we define and briefly explore the operator weakly $p$-compact matrix sets and corresponding mappings. The latter enjoy factorization properties (Theorem \ref{thm-factorization-weakly-p-compact}, Proposition \ref{prop-Wp-is-Gamma-p-sur}) similar to the ones already known for the operator $p$-compact ones \cite{ChaDiGa-Operator-p-compact}.
Moreover, in Section \ref{sect-completely-summing} we show that this class is intimately related to Pisier's completely $p$-summing maps: the latter send relatively weakly $p$-compact matrix sets to relatively operator $p$-compact matrix sets.
In the Banach case, $p$-summing maps are characterized as those which map relatively compact sets to relatively $p$-compact sets. In the noncommutative setting, while one implication always holds for the other one we need a technical condition which is dual to exactness (See Theorem \ref{thm-adjoint-of-p-summing}).

Section \ref{sec-regularity} studies the question of regularity for the mapping ideal of operator $p$-compact maps, that is, the question of whether a mapping $T : V \to W$ is operator $p$-compact if and only if so is its composition with the canonical embedding $W \to W''$. While the answer is positive for the classical $p$-compact mappings on Banach spaces, in the noncommutative setting we can only prove an analogous result under the additional hypotheses of having $N$-maximal domain and locally reflexive codomain (Corollary \ref{cor-compactness-in-the-bidual-for-operators}).

Just from the definition, every operator $p$-compact mapping is naturally associated to a relatively $p$-compact matrix set. In the last section we prove the reverse: given a relatively operator $p$-compact matrix set, can we associate to it an operator $p$-compact mapping (Theorem \ref{thm:equivalences-p-compact-u}). In the classical situation, there is an additional equivalence with the $p$-nuclearity of the adjoint of the map associated to a $p$-compact set. In order to obtain said equivalence in the noncommutative setting (Theorem \ref{thm:equivalences-p-compact-j}), we need technical conditions closely related to the ones in the previous paragraph.

We remark that in the classical setting, the analogues of the properties examined herein serve as the foundation for investigating an associated approximation property, the $p$-approximation property introduced in \cite{Sinha-Karn}. This paper aims to analogously  provide tools which we will use to develop corresponding approximation properties within the non-commutative framework, in the parallel manuscript \cite{CDDG-pOAP}.

\section{Notation and preliminaries}\label{sec-preliminaries}

We only assume familiarity with the basic theory of operator spaces; the books \cite{Pisier-OS-theory} and \cite{Effros-Ruan-book} are excellent references. Our notation follows closely that from \cite{Pisier-Asterisque, Pisier-OS-theory}, with one notable exception: we denote the dual of a space $V$ by $V'$.

Throughout the article, $V$ and $W$ denote operator spaces. For each $n$, $M_n(V)$ represents the space of $n\times n$ matrices with elements from $V$. We denote the $n$-amplification of a linear mapping $T:V\to W$ as $T_n:M_n(V)\to M_n(W)$. The space of completely bounded linear mappings from $V$ to $W$ is denoted by $\CB(V,W)$, with the subspace of finite-rank mappings represented by $\mathcal{F}(V,W)$. Our notation for the minimal and projective operator space tensor products is respectively $\otimes_{\min}$ and $\otimes_{\proj}$. The canonical completely isometric embedding into the bidual is denoted by $\iota_V : V \to V''$.

A linear map $Q : V \to W$ between operator spaces is called a complete 1-quotient  if it is onto and the associated map from $V/\ker(Q)$ to $W$ is a completely isometric isomorphism.
In \cite[Sec. 2.4]{Pisier-OS-theory}, these maps are referred to as complete metric surjections. It is proven therein that a linear map $T:V \to W$ is a complete 1-quotient if and only if its adjoint $T' : W' \to V'$ is a completely isometric embedding.

 We say that a matrix set $\mathbf{K}$ is open (resp. closed) whenever each $K_n$ is open (resp. closed). By the closure of a matrix set $\mathbf{K} = (K_n)_n$ we mean taking the closure on each level: $\overline{\mathbf{K}} = ( \overline{K_n})_n$. 

We say that matrix set $\mathbf{K}$ over $V$ is \emph{completely bounded} if there exists a constant $C>0$ such that $\mathbf{K} \subseteq C \mathbf{B}_V$. 

To simplify certain statements, whenever $\mathbf{K} = (K_n)_n$ is a matrix set, we will use the shorthand ``$x \in \mathbf{K}$'' to mean ``there exists $n \in \mathbb{N}$ such that $x \in K_n$'' . Also, for $T\in\CB(V,W)$, the expression $T(\mathbf{K})$ denotes the matrix set $\big(T_n(K_n)\big)_n$.

The $\ell_p$ spaces are essential in defining and studying $p$-compactness in Banach spaces. The noncommutative analog of $\ell_p$ is the Schatten class $\mathcal{S}_p$. For $1 \le p < \infty$, $\mathcal{S}_p$ comprises all compact mappings $T$ on $\ell_2$ such that $\tr |T|^p < \infty$, equipped with the norm $\n{T}_{\mathcal{S}_p} = \big( \tr |T|^p \big)^{1/p}$. For $p=\infty$, $\mathcal{S}_\infty$ denotes the space of all compact mappings on $\ell_2$ with the operator norm.
The analogy  is shown by identifying $\ell_p$  with the diagonal mappings in $\mathcal{S}_p$,  noting that any two diagonal mappings commute.
The operator space structure on $\mathcal{S}_p$ is provided through Pisier's theory of complex interpolation for operator spaces \cite[Sec. 2]{Pisier-interpolation},  \cite[Sec. 2.7]{Pisier-OS-theory}. Note that
 $\mathcal{S}_\infty$ has a canonical operator space structure as it is a $C^*$-algebra  \cite[p. 21]{Effros-Ruan-book}. Also, the space $\mathcal{S}_1 = \mathcal{S}_\infty'$ naturally inherits an operator space structure through duality (as usual, the duality pairing is defined by $\langle a, b \rangle := \tr(a^t b)$). Using this, it is possible to endow each intermediate space $\mathcal{S}_p$ with a well-defined operator space structure.
As remarked in \cite[p. 141]{Pisier-OS-theory}, this abstract approach realizes $\mathcal{S}_p$ as a subspace of a $\mathcal B(H)$ space in a highly nonstandard way.
More generally, an operator space $V$ yields a $V$-valued version of $\mathcal{S}_p$, denoted by $\mathcal{S}_p[V]$:
$\mathcal{S}_\infty[V]$ is the minimal operator space tensor product of $\mathcal{S}_\infty$ and $V$,
$\mathcal{S}_1[V]$ is the operator space projective tensor product of $\mathcal{S}_1$ and $V$,
and once again in the case $1 < p < \infty$ we define $\mathcal{S}_p[V]$ via complex interpolation between $\mathcal{S}_\infty[V]$ and $\mathcal{S}_1[V]$  \cite{Pisier-Asterisque}.
For $1 < p \le \infty$, the dual of $\mathcal{S}_p[V]$ can be canonically identified with $\mathcal{S}_{p'}[V']$, where $p'$ satisfies $1/p +  1/p' = 1$ \cite[Cor. 1.8]{Pisier-Asterisque}.
In the discussion above, if we replace $\mathcal{S}_1$ by the space $\mathcal{S}_1^n$ of $n\times n$ matrices with the trace norm, and $\mathcal{S}_\infty$ by the space $M_n$, we can analogously construct operator spaces $\mathcal{S}_p^n$ and $\mathcal{S}_p^n[V]$.
We will often consider the elements of the spaces $\mathcal{S}_p \widehat\otimes_{\min} V$ and $\mathcal{S}_p[V]$ as infinite matrices with entries in $V$.
In the first case the meaning is clear: since $\mathcal{S}_p \widehat\otimes_{\min} V$ completely isometrically embeds into $\CB(\mathcal{S}_p',V)$, we identify $v \in \mathcal{S}_p \widehat\otimes_{\min} V$ with the infinite $V$-valued matrix that arises from applying $v$ (considered as a map $\mathcal{S}_p' \to V$) to the matrix units in $\mathcal{S}_p'$.
For $v \in \mathcal{S}_p[V]$ the identification as an infinite matrix is not immediately clear, since $\mathcal{S}_p[V]$ was constructed using complex interpolation.
The reader is invited to check \cite[pp. 18--20]{Pisier-Asterisque} for further details.

We use $M(a,b)$ to denote the two-sided multiplication mapping $x \mapsto axb$.

An operator space $V$ is said to be \emph{locally reflexive} if for any finite-dimensional operator space $W$, every complete contraction $T : W \to V''$ is the point-weak$^*$ limit of a net of complete contractions $T_\alpha : W \to V$.
The operator space $V$ is said to be \emph{strongly locally reflexive} if given finite-dimensional subspaces $F\subseteq V''$ and $N \subseteq V'$, and $\varepsilon>0$, there exists a complete isomorphism $T : F \to E \subseteq V$ such that
(a) $\n{T}_{\cb},\n{T^{-1}}_{\cb} < 1+\varepsilon$,
(b) $\pair{Tv}{v'} = \pair{v}{v'}$ for all $v \in F$ and $v' \in N$,
(c) $Tv=v$ for all $v\in F \cap V$.

The operator space structure of an $\ell_\infty$-sum of operator spaces $\ell_\infty(\{V_i\}_{i \in  I})$ is given by the identification $M_n(\ell_\infty(\{V_i\}_{i \in  I})) = \ell_\infty(\{M_n(V_i)\}_{i \in  I})$.
The operator space structure of an $\ell_1$-sum of operator spaces $\ell_1(\{V_i\}_{i \in  I})$ is characterized by the following universal property: for any operator space $W$ and any linear map $T : \ell_1(\{V_i\}_{i \in  I}) \to W$, $\n{T}_{\cb} \le 1$ if and only if for all $i\in I$ we have $\n{TJ_i}_{\cb} \le 1$, where $J_i : V_i \to \ell_1(\{V_j\}_{j \in  I})$ is the canonical injection (that is, the map sending $v \in V_i$ to the vector having $v$ in the $i$-th position and 0 everywhere else).
In the case where $V_i = W$ for all $i\in I$, we use the shorthands
$\ell_\infty(I ; W) = \ell_\infty(\{V_i\}_{i \in  I})$ and
$\ell_1(I ; W) = \ell_1(\{V_i\}_{i \in  I})$.

A \emph{mapping ideal} $(\mathfrak{A},\|\cdot\|_{\mathfrak{A}})$ is an assignment, for each pair of operator spaces $V,W$, of a linear space $\mathfrak{A}(V,W) \subseteq \CB(V,W)$  together with an operator space structure $\|\cdot\|_{\mathfrak{A}}$ on $\mathfrak{A}(V,W)$ such that
\begin{enumerate}[(a)]
\item The identity map $\mathfrak{A}(V,W) \to \CB(V,W)$ is a complete contraction.
\item For every $v'\in M_n(V')$ and $w\in M_m(W)$ the mapping $v'\otimes w$ belongs to $M_{nm}(\mathfrak{A}(V,W))$ and $\|v'\otimes w\|_{\mathfrak{A}}=\|v'\|_{M_n(V')} \|w\|_{M_m(W)}$.
\item The ideal property: whenever $T \in M_n(\mathfrak{A}(V,W))$, $R \in \CB(V_0,V)$ and $S \in \CB(W,W_0)$, it follows that $S_n \circ T \circ R$ belongs to $ M_n(\mathfrak{A}(V_0,W_0))$ with
$$
\| S_n \circ T \circ R \|_{\mathfrak{A}} \le \n{S}_{\cb} \|T\|_{\mathfrak{A}} \n{R}_{\cb}.
$$
\end{enumerate}
Note that this is the definition of \cite[Def. 7.1]{ChaDiGa-tensor-norms}, which is stronger than that of \cite[Sec. 12.2]{Effros-Ruan-book} (because of the item (b)).
All of the mapping ideals considered in the present paper have been checked to satisfy this stronger definition, see \cite[Sec. 7]{ChaDiGa-tensor-norms}.

Recall that an operator space $V$ is said to have the \emph{Completely Metric Approximation Property (CMAP)} if there exists a net of finite rank complete contractions in $\CB(V,V)$ that converges pointwise to the identity of $V$.

The following result will be useful later.

\begin{lemma}\label{lemma-ell_infty-has-CMAP}
Let $N\in \N$ be such that for every $i\in I$, $n_i \in \N$ satisfies $n_i \le N$. 
Then $\ell_\infty(\{M_{n_i}\}_{i \in I})$ has CMAP.
\end{lemma}

\begin{proof}
Since $\ell_\infty(I)$ has the MAP there exists a net of finite rank contractions converging pointwise to the identity. Using the minimality of $\ell_\infty(I)$ these operators are in fact complete contractions. Thus, by tensorizing with the identity of $M_N$ we obtain the CMAP for $\ell_\infty(I) \widehat\otimes_{\min} M_N = \ell_\infty(I; M_N)$.
Now, $\ell_\infty(\{M_{n_i}\}_{i \in I})$ is completely contractively complemented in $\ell_\infty(I; M_N)$, so it also has CMAP.
\end{proof}

\section{Operator $p$-compact sets and mappings}\label{sec-sets-and-mappings}

In this section, we develop a ``geometric" perspective on the concept of $p$-compactness for matrix sets and introduce a quantitative measure of this notion. Building on this, we establish a relationship between the norm of the ideal in terms of the measure.

In \cite[Def. 3.9]{ChaDiGa-Operator-p-compact} we define the notion of $p$-compactness for a matrix set. Now we will provide a different perspective which is equivalent to the previous one.

\begin{definition}\label{defn-operator-p-compact}
   A matrix set $\mathbf{K}=(K_n)$ over $V$ is \emph{relatively operator $p$-compact} if there exists $v \in \mathcal{S}_p[V]$ such that $\mathbf{K}\subset \Theta^v(\mathbf{B}_{\mathcal{S}_p'})$, where $\Theta^v: \mathcal{S}_p' \to V$ is defined by $(\alpha_{ij}) \mapsto \sum_{i,j=1}^{\infty} \alpha_{ij} v_{ij}$. 
   \end{definition}

It might not be immediately apparent that the maps $\Theta^v : \mathcal{S}_p' \to V$ associated to $v\in \mathcal{S}_p[V]$ appearing in Definition \ref{defn-operator-p-compact} above are even well-defined. Let us take a moment to confirm that they make sense.

\begin{lemma}\label{lemma-Theta-well-defined}
The formal identity $\mathcal{S}_p[V] \to \mathcal{S}_p\widehat\otimes_{\min} V \hookrightarrow \CB(\mathcal{S}_p',V)$ is a completely contractive injection. Moreover, the image of $v \in \mathcal{S}_p[V]$ is the map $\Theta^v : \mathcal{S}_p' \to V$ given by $(\alpha_{ij}) \mapsto \sum_{i,j=1}^{\infty} \alpha_{ij} v_{ij}$.
\end{lemma}

\begin{proof}
For $\theta \in  [0,1]$ denote $\mathcal{R}(\theta) = ( \mathcal{R}, \mathcal{C} )_\theta$, where $\mathcal{R}$ and $\mathcal{C}$ are the row and column operator spaces, respectively.
By \cite[Thm. 1.1]{Pisier-Asterisque} we have for any $1\le p\le \infty$, completely isometric identifications
\[
\mathcal{S}_p[V] = \mathcal{R}(1/p') \widehat\otimes_h V \widehat\otimes_h \mathcal{R}(1/p), \qquad \mathcal{S}_p\widehat\otimes_{\min} V = \big( \mathcal{R}(1/p')\widehat\otimes_h \mathcal{R}(1/p)  \big) \widehat\otimes_{\min} V.
\]
Now, since the Haagerup tensor product dominates the minimal one, and using the metric mapping property of the Haagerup tensor product, we have that the formal identity below is a complete contraction
\[
\mathcal{R}(1/p') \widehat\otimes_h V \widehat\otimes_h \mathcal{R}(1/p) \to \big( \mathcal{R}(1/p') \widehat\otimes_{\min} V \big)  \widehat\otimes_h \mathcal{R}(1/p).
\]
By the ``tensor shuffle'' \cite[Thm. 5.15]{Pisier-OS-theory}, the shuffle map
\[
\big( \mathcal{R}(1/p') \widehat\otimes_{\min} V \big)  \widehat\otimes_h \mathcal{R}(1/p) \to \big( \mathcal{R}(1/p')\widehat\otimes_h \mathcal{R}(1/p)  \big) \widehat\otimes_{\min} V
\]
is also a complete contraction.
Composing the aforementioned two maps yields that the formal identity $\mathcal{S}_p[V] \to \mathcal{S}_p\widehat\otimes_{\min} V$ is a complete contraction as well.

Note that for a $v=(v_{ij}) \in S_p[V]$ which is supported on the initial $n \times n$ block, the associated operator given by the standard inclusion $\mathcal{S}_p\widehat\otimes_{\min} V \hookrightarrow \CB(\mathcal{S}_p',V)$ is precisely the one given by $(\alpha_{ij}) \mapsto \sum_{i,j=1}^n \alpha_{ij} v_{ij}$.
Since any element of $\mathcal{S}_p[V]$ (respectively $\mathcal{S}_p\widehat\otimes_{\min} V$) is the norm-limit of its truncations, it follows that $v \in \mathcal{S}_p[V]$ corresponds to the map $\Theta^v : \mathcal{S}_p' \to V$ given by $(\alpha_{ij}) \mapsto \lim_{n\to\infty} \sum_{i,j=1}^{n} \alpha_{ij} v_{ij}$.
The injectivity is immediate, since if $v = (v_{ij})_{i,j=1}^\infty \in \mathcal{S}_p[V]$ corresponds to the zero element in $\mathcal{S}_p\widehat\otimes_{\min} V$, it is clear that we must have $v_{ij}=0$ for all $i,j$.
\end{proof}

For simplicity, from now on, we will say that a matrix set is operator compact if it is an operator  $\infty$-compact matrix set. This notion was coined by Webster in his thesis \cite[Def. 4.1.1]{Webster}.

As already mentioned, in \cite{ChaDiGa-Operator-p-compact} the concept introduced in Definition \ref{defn-operator-p-compact} was presented in a different way: A matrix set $\mathbf{K}$ is relatively operator $p$-compact if it is contained in $\overline{\mathbf{co}_p(v)}$ for some $v\in \mathcal{S}_p [V]$, where for any $n \in \mathbb{N}$:
\[
\left(\overline{\mathbf{co}_p(v)}\right)_n = \left\{(\sigma \otimes \text{Id}_V)v \in M_n(V) \,:\, \sigma \in M_n(\mathcal{S}_{p}'), \; \lVert \sigma \rVert_{M_n(\mathcal{S}_{p}')} \leq 1  \right\},
\]
and $\sigma \otimes \text{Id}_V$ denotes the tensor product of the matrix $\sigma$ with the identity mapping.
The equivalence between this and Definition \ref{defn-operator-p-compact} becomes apparent when interpreting $\sigma \in M_n(\mathcal{S}_{p}')$ as a mapping from $\mathcal{S}_p$ to $M_n$ using the standard identification $ M_n(\mathcal{S}_p') = M_n(\CB(\mathcal{S}_p,\C)) = \CB(\mathcal{S}_p,M_n)$. Hence, $\overline{\mathbf{co}_p(v)}=\Theta^v(\mathbf B_{\mathcal{S}_{p}'})$.

In the context of Banach spaces, the measure of $p$-compactness for a set is introduced in \cite[Def. 2.1]{Galicer-Lassalle-Turco}. Here, we extend this concept to the non-commutative realm.

\begin{definition}\label{def-measure-of-p-compactness}
Let $1 \le p \le \infty$. For a relatively operator $p$-compact matrix set $\mathbf{K}$ over an operator space $V$, we define its \emph{measure of operator $p$-compactness} by
\[
\frak m_p^o(\mathbf{K}) = \inf \big\{ \n{v}_{\mathcal{S}_p[V]} \;:\; v \in \mathcal{S}_p[V],  \mathbf{K}\subset \Theta^v(\mathbf{B}_{\mathcal{S}_p'})  \big\}.
\]
\end{definition}

We first note that for $v \in \mathcal{S}_p[V]$, the matrix set $\Theta^v(\mathbf{B}_{\mathcal{S}_p'})$ is closed. Consequently, the property of being relatively $p$-compact is preserved under closure as stated in lemma below.  

To see this, fix $n \in \mathbb{N}$ and consider the map  
\[
(\Theta^v)_n: M_n(\mathcal{S}_{p}') \to M_n(V), \quad \sigma \mapsto (\sigma \otimes \mathrm{id})(v).
\]  
We aim to show that the image of $B_{M_n(\mathcal{S}_{p}')}$ under $(\Theta^v)_n$ is also closed.  

Let $(\sigma_l)_l \subseteq B_{M_n(\mathcal{S}_{p}')}$ be a sequence such that $(\Theta^v)_n(\sigma_l) \to \bar{v} \in M_n(V)$. We need to show that $\bar{v} = (\Theta^v)_n(\sigma)$ for some $\sigma \in M_n(\mathcal{S}_{p}')$ with $\|\sigma\|_{M_n(\mathcal{S}_{p}')} \leq 1$.  

Since $B_{M_n(\mathcal{S}_{p}')}$ is weak$^*$ sequentially compact (as the predual of $M_n(\mathcal{S}_{p}')$ is separable), there exists a subsequence $(\sigma_{l_j})_j$ that converges weak$^*$ to an element $\sigma \in B_{M_n(\mathcal{S}_{p}')}$. By the uniqueness of limits, the proof will be complete once we establish that, for every matrix $v' \in M_n(V')$, the scalar pairing (in the sense of \cite[1.1.24]{Effros-Ruan-book}) satisfies:  
\[
\langle (\Theta^v)_n(\sigma_{l_j}), v' \rangle \to \langle (\Theta^v)_n(\sigma), v' \rangle.
\]  

Now, observe that:  
\[
\langle (\Theta^v)_n(\sigma_{l_j}), v' \rangle = \langle (\sigma_{l_j} \otimes \mathrm{id})v, v' \rangle = \langle \sigma_{l_j}, (\mathrm{id} \otimes v')v \rangle.  
\]  
By the weak$^*$ convergence of $\sigma_{l_j}$ to $\sigma$, we have:  
\[
\langle \sigma_{l_j}, (\mathrm{id} \otimes v')v \rangle \to \langle \sigma, (\mathrm{id} \otimes v')v \rangle = \langle (\Theta^v)_n(\sigma), v' \rangle.
\]

\begin{lemma}\label{lemma-p-compact-closure}
Let $1 \le p \le \infty$.
A matrix set $\mathbf{K}$ over an operator space $V$ is relatively operator $p$-compact if and only if so is $\overline{\mathbf{K}}$. Furthermore, $\frak m_p^o(\mathbf{K}) = \frak m_p^o(\overline{\mathbf{K}})$.
\end{lemma}

In the context of Banach spaces, a $p$-compact map is defined as a mapping that sends the unit ball into a relatively $p$-compact set. Extending this concept to our non-commutative framework, a similar property would involve mapping the matrix unit ball into a relatively operator $p$-compact matrix set. We recall this definition that was already introduced in \cite{ChaDiGa-Operator-p-compact} with the usual terminology of that article.

\begin{definition} \label{defn: operator p-compact maps}
A completely bounded mapping  $T: V \to W$  is operator $p$-compact ($1\le p\le \infty$) if $T(\mathbf{B}_V)$ is a relatively operator $p$-compact matrix set in $W$. The operator $p$-compact norm of $T$ is defined as   
\begin{equation}\label{kpo}
\kappa_p^o(T) = \frak m_p^o( T(\mathbf{B}_V) ).
\end{equation}

\end{definition}

We denote the class of all operator $p$-compact mappings $T:V \to W$ by $\mathcal K_p^o(V,W)$,  and this is a mapping ideal.
For more properties and equivalences related to this, see \cite{ChaDiGa-Operator-p-compact}.
Again, we will abbreviate the terminology by saying that a mapping is operator compact if it is operator $\infty$-compact.

We recall the following prototypical examples of 
 operator $p$-compact maps from \cite[Prop. 2.7]{ChaDiGa-Operator-p-compact} that will be repeteadly used along the article.

\begin{example}\label{example-multiplication-is-operator-compact}
Let $1 \leq p \leq \infty$ and  $a,b \in \mathcal{S}_{2p}$. Then the multiplication mapping
$ M(a,b) : \mathcal{S}_{p'} \to \mathcal{S}_1$ given by $x \mapsto a \cdot x \cdot b$
is operator $p$-compact and satisfies $\kappa_p^o(M(a,b)) = \n{a}_{\mathcal{S}_{2p}} \n{b}_{\mathcal{S}_{2p}}$. 
We remark that in the case $p=1$, Lemma \ref{lemma:multiplication-operator-on-B(ell_2)} yields the same conclusion for the multiplication mapping $M(a,b) : \mathcal{B}(\ell_2) \to \mathcal{S}_1$.
In particular, for $a, b \in \mathcal{S}_\infty$, $M(a,b) : \mathcal{S}_1 \to \mathcal{S}_1$ is operator compact.
\end{example}

In the literature, we encounter other notions of compactness within the framework of operator spaces: there are several developed by Webster in his thesis and there is another provided by Yew. Next, we focus  on relating these definitions to the one we just presented.

Let us begin by comparing our noncommutative notion of compactness with the  one from \cite{Yew}, which is defined as follows:

\begin{definition}\label{defn-Yew-compact}
Let $Z$ be an operator space. A linear map $T:V \to W$ is said to be \emph{$Z$-compact} if $T(\mathbf{B}_V) \subset \Theta^w(\mathbf{B}_{Z'})$, where $w \in Z \widehat{\otimes}_{\min}W$ and $\Theta^w : Z' \to W$ is the associated operator.
The $Z$-compact norm of $T$ is
\[
\kappa_Z(T) = \inf\big\{ \n{w}_{Z \widehat{\otimes}_{\min}W} \;:\; w \in Z \widehat{\otimes}_{\min}W, \; T(\mathbf{B}_V) \subset \Theta^w(\mathbf{B}_{Z'}) \big\}
\]
\end{definition}

Comparing Definitions \ref{defn-operator-p-compact} and \ref{defn-Yew-compact},
it is clear that our operator $p$-compact mappings are $\mathcal{S}_p$-compact in the sense of Yew, given that our condition $v\in \mathcal{S}_p[V]$ is more stringent than Yew's requirement of $v \in \mathcal{S}_p \widehat{\otimes}_{\min} V$ (see Lemma \ref{lemma-Theta-well-defined}). For $p=\infty$ both notions coincide since $\mathcal{S}_\infty[V]=\mathcal{S}_\infty\widehat\otimes_{\min} V$.

Although Yew's approach is more general, our more restricted formulation in the case of $\mathcal{S}_p$ yields enhanced properties. While the  $Z$-compact maps in \cite{Yew} do not in general constitute a vector space, our operator $p$-compact maps not only form a vector space but also naturally have the structure of an operator space \cite{ChaDiGa-Operator-p-compact}. Additionally, we obtained a more refined factorization theorem, as evidenced by the comparison between \cite[Thm. 3.9]{Yew} and \cite[Prop. 3.8]{ChaDiGa-Operator-p-compact}.

Returning to the realm of Banach spaces, the property of a set being compact (i.e., the standard notion) has several equivalently useful formulations. However, when this concept is extended to the non-commutative context, distinct and non-equivalent definitions emerge.

One such extension is the concept of operator compact matrix sets, which we have been working with. Another is a notion defined by Webster in \cite[Def. 4.1.2]{Webster}, building on prior work by Saar \cite{Saar}.

\begin{definition}
    A matrix set $\mathbf{K}=(K_n)$ over $V$ is \emph{ completely compact} if it is closed, completely bounded  and for every $\varepsilon>0$ there exists a finite-dimensional $V_\varepsilon \subseteq V$ such that for each $n\in\N$ and $x\in K_n$ there is  $v\in M_n(V_\varepsilon) $ with $\|x-v\|\le \varepsilon$.
    A linear mapping $T : V \to W$ between operator spaces is called \emph{completely compact} if $\overline{T(\mathbf{B}_V)}$ is completely compact.
\end{definition}

 It is worth noting, as was observed in \cite{Webster}, that this definition is more general; that is to say, if a mapping is operator compact, it is also completely  compact.

The following  example shows that the prototypical mapping, involving the two-sided multiplication by elements of $\mathcal{S}_{\infty}$ over $\mathcal{S}_p$, is always completely compact (but generally not operator compact when $1< p \le \infty$).  
We recall a useful definition relevant to our discussion. For a finite-dimensional operator space $V$, its \emph{exactness constant} is defined as  
\[
\ex(V) = \inf \big\{ \|T\|_{\cb} \|T^{-1}\|_{\cb} : W \subseteq M_n, \, T : V \to W \text{ is an isomorphism} \big\}.
\]  

\begin{proposition}\label{proposition-multiplication-is-completely-compact}
Let $1 \leq p \leq \infty$ and  $a,b \in \mathcal{S}_{\infty}$. Then the multiplication mapping
$
M(a,b) : \mathcal{S}_p \to \mathcal{S}_p
$
is completely compact.
However, when $1< p \le \infty$ it is not necessarily operator compact. 
\end{proposition}

\begin{proof}
It is easy to check from the definition that if a linear map is a limit in the $\cb$-norm of finite-rank maps, then it is completely compact.
This is the case for $M(a,b)$ which can be shown by approximating $a$ and $b$ with their truncations, see e.g. the proof of \cite[Prop. 2.7]{ChaDiGa-Operator-p-compact}.

With respect to the second statement, if $I_n \in M_n$ denotes the identity matrix, note that by \cite[Cor. 4.16]{Oikhberg-Thesis}
\[
\kappa_\infty^o( M(I_n,I_n) : \mathcal{S}_p^n \to \mathcal{S}_p^n ) = \kappa_\infty^o( Id_{\mathcal{S}_p^n}) = \ex(\mathcal{S}_{p'}^n) 
\]
which converges to infinity when $1 < p \le \infty$ by \cite[Ex. 3.3.1.3]{Junge-Habilitationschrift}.
Now, consider a sequence of positive numbers $\alpha_n >0$ converging to 0 and such that $\alpha_n^2\ex(\mathcal{S}_{p'}^n)$ converges to infinity (for example, $\alpha_n = \ex(\mathcal{S}_{p'}^n)^{-1/4}$).
Note that $a= b = \oplus_{n=1}^\infty \alpha_n I_n \in \mathcal{S}_\infty$ because $\alpha_n \to 0$.
On the other hand, for each $n \in \N$ we will have
\[
\kappa_\infty^o( M(a,b):\mathcal{S}_p \to \mathcal{S}_p ) \ge \kappa_\infty^o( M(\alpha_nI_n,\alpha_nI_n) : \mathcal{S}_p^n \to \mathcal{S}_p^n ) = \alpha_n^2 \ex(\mathcal{S}_{p'}^n) \xrightarrow[n\to\infty]{} \infty.
\]
Note that the inequality above is justified because $\mathcal{S}_p^n$ is completely contractively complemented in $\mathcal{S}_p$. 
\end{proof}

\begin{remark}
Let us observe that the case $p=\infty$ in the previous proof gives a negative answer to a question of Webster \cite[Sec. 4.1]{Webster} restated in \cite[Problem 4.3]{ChaDiGa-Operator-p-compact}: whether the space of operator compact maps must be closed in the $\cb$-norm.
Indeed, the multiplication mapping $M(a,b):\mathcal{S}_\infty \to \mathcal{S}_\infty$ appearing in the aforementioned proof is not operator compact, but it is the $\cb$-norm limit of its truncations which have finite rank and are therefore operator compact.
It is interesting to note that all the ingredients needed for this answer to the question, posed in 1997, were already available in 1998 (which we did not realize when we restated the question in \cite{ChaDiGa-Operator-p-compact}).
\end{remark}

\subsection{Monotonicity of operator $p$-compactness}

Similar to the classical context, $p$-compactness for matrix sets is monotonic in $p$. To establish this, we will first prove a lemma which illustrates how, given a suitable factorization of an element $v$ in $\mathcal{S}_p[V]$, one can decompose the mapping $\Theta^v$ in terms of a canonical multiplication map. This insight proves to be highly beneficial throughout our study.
It is noteworthy to emphasize that the referred factorization of $v$ is guaranteed to exist by \cite[Lem. 4.2]{CD-completely-mixing} (which is a generalization of \cite[Thm. 1.5]{Pisier-Asterisque}).

\begin{lemma}\label{lemma-factorization-Pisier}
Let $1 \le p , q, r \le \infty$ with $1/p = 1/q + 1/r$ and let $V$ be an operator space.
Suppose  that $v \in \mathcal{S}_p[V]$ is written as $v = a \cdot u \cdot b$ with $a,b \in \mathcal{S}_{2r}$ and $u \in \mathcal{S}_q[V]$. Then the associated map $\Theta^v : \mathcal{S}_{p}' \to V$ factors as $\Theta^u \circ M(a^t, b^t)$ where $\Theta^u : \mathcal{S}_{q}' \to V$ is the mapping associated to $u$ and $M(a^t, b^t) : \mathcal{S}_{p}' \to \mathcal{S}_{q}'$ is the two-sided multiplication mapping.
\end{lemma}

\begin{proof}
We start by considering the case where $v$, $a$, $b$ and $w$ are all finitely supported on the initial $n \times n$ block.
Then $\Theta^v : \mathcal{S}_{p'} \to V$ (that is, the map corresponding to $v$ as an element of $\mathcal{S}_p \otimes_{\min} V$) is given by $(\alpha_{ij}) \to \sum_{i,j=1}^n \alpha_{ij}v_{ij}$, and similarly $\Theta^u : \mathcal{S}_{q'} \to V$ is the map $(\alpha_{ij}) \to \sum_{i,j=1}^n \alpha_{ij}u_{ij}$.
We denote by $E_{ij}$ the matrix units, i.e. $E_{ij}$ has a 1 in the $(i,j)$-th position and zeros elsewhere.
On one hand it is clear that for $1\le i, j \le n$ we have $\Theta^v(E_{ij}) = v_{ij} = \sum_{k,\ell=1}^n a_{ik} u_{k\ell} b_{\ell j}$.
On the other hand,
\begin{align*}
\Theta^u \circ M(a^t, b^t) E_{ij} &  = \Theta^u(a^t E_{ij} b^t) = \Theta^u \sum_{k,\ell=1}^n (a^t)_{ki} (b^t)_{j\ell} E_{k \ell} \\
& = \Theta^u \sum_{k,\ell=1}^n a_{ik} b_{\ell j} E_{k \ell} 
= \sum_{k,\ell=1}^n a_{ik} b_{\ell j} \Theta^uE_{k \ell} = \sum_{k,\ell=1}^n a_{ik} b_{\ell j} u_{k \ell},
\end{align*}
which proves that $\Theta^v = \Theta^u \circ M(a^t, b^t)$.

The above argument shows the desired result in the finitely supported case, so now we just need to make sure that things work well in the limit.
For each $n\in\N$, let $u[n],a[n],b[n]$ be the truncations of $u, a, b$ respectively to the initial $n\times n$ block (but still considered as infinite matrices). We let $\widetilde{v}[n] = a[n]\cdot u[n] \cdot b[n] = a[n] \cdot u \cdot b[n]$.
From the finitely supported case, note that
$\Theta^{\widetilde{v}[n]} = \Theta^{u[n]} \circ M(a[n]^t,b[n]^t)$.
Since $a[n]^t \to a^t$ and $b[n]^t \to b^t$ in $\mathcal{S}_{2r}$, by \cite[Lem. 2.4]{Oikhberg} we have that $M(a[n]^t,b[n]^t) \to M(a^t,b^t)$
 in $\CB(\mathcal{S}_p',\mathcal{S}_q')$.
 Since $u[n] \to u$ in $\mathcal{S}_q[V]$, we have from Lemma \ref{lemma-Theta-well-defined} that $\Theta^{u[n]} \to \Theta^u$ in $\CB(\mathcal{S}_q',V)$.
 Similarly,  \cite[Thm. 1.5]{Pisier-Asterisque} implies that $\widetilde{v}[n] \to v$ in $\mathcal{S}_p[V]$ and thus $\Theta^{\widetilde{v}[n]} \to \Theta^v$ in $\CB(\mathcal{S}_p',V)$. All together, these imply  $\Theta^v = \Theta^u \circ M(a^t, b^t)$ as desired.
 \end{proof}

We now present a comparative result that emulates the classical relationship between  $p$-compactness and $q$-compactness in Banach spaces.

\begin{proposition}\label{prop-monotonicity-compactness}
Let $1 \le p < q \le \infty$.
Every relatively operator $p$-compact matrix set $\mathbf{K}$ is relatively operator $q$-compact with $\frak m_q^o(\mathbf{K})\le \frak m_p^o(\mathbf{K})$. 
Consequently, every operator $p$-compact mapping $T:V \to W$ is operator $q$-compact with $\kappa_q^o(T)\le \kappa_p^o(T)$.
\end{proposition}

\begin{proof}
Let $1/p = 1/q+1/r$.
Let $\mathbf{K}$ be a relatively operator $p$-compact matrix set over the operator space $W$.
Let $w \in \mathcal{S}_p[W]$ such that $\mathbf{K} \subseteq \Theta^w( \mathbf{B}_{\mathcal{S}_{p}'})$.
By \cite[Lem. 4.2]{CD-completely-mixing}, given $\varepsilon>0$ we can factor $w = a \cdot z \cdot b$ with $a,b \in \mathcal{S}_{2r}$, $z \in \mathcal{S}_q[W]$, $\n{a}_{2r} = \n{b}_{2r} =1$ and $\n{z}_{\mathcal{S}_q[W]} \le (1+\varepsilon)\n{w}_{\mathcal{S}_p[W]}$.
By Lemma \ref{lemma-factorization-Pisier}, $\Theta^w = \Theta^z \circ M(a^t, b^t)$.
Note that by \cite[Thm. 2.1]{Oikhberg} $M(a^t,b^t) : \mathcal{S}_{p}' \to \mathcal{S}_{q}'$ is a complete contraction, since $1/r = 1/p - 1/q = 1/q'-1/p'$, so
\[
\mathbf{K}  \subseteq \Theta^w( \mathbf{B}_{\mathcal{S}_{p}'} ) = \Theta^z \circ M(a^t, b^t)( \mathbf{B}_{\mathcal{S}_{p}'} ) \subseteq \Theta^z ( \mathbf{B}_{\mathcal{S}_{q}'} ),
\]
which implies that $\mathbf{K}$ is relatively operator $q$-compact and moreover $\frak m_q^o(\mathbf{K}) \le \n{z}_{\mathcal{S}_q[W]} \le (1+\varepsilon) \n{w}_{\mathcal{S}_p[W]}$.
Taking the infimum over $w$ and $\varepsilon$ yields $\frak m_q^o(\mathbf{K})\le \frak m_p^o(\mathbf{K})$.
\end{proof}

Later on, we will demonstrate in Corollary \ref{cor: p-not-q} that there are matrix sets (and completely bounded maps) which are operator $p$-compact but not operator $q$-compact for $q < p$.

\subsection{Factorization of operator $p$-compact mappings}

In the classical setting, the following trick is oftentimes useful: a sequence $(x_n)$ in $\ell_p$ can be written as $(\lambda_ny_n)$ where  $(y_n) \in \ell_p$, $(\lambda_n)\in c_0$ and $\n{(y_n)}_{\ell_p}\n{(\lambda_n)}_{c_0} \le (1+\varepsilon)\n{(x_n)}_{\ell_p}$.
In the noncommutative setting, a similar maneuver holds. We record this  fact as a remark for future reference. Note that it follows from the aforementioned classical situation together with the singular value decomposition.

\begin{remark}\label{remark-trick}
Let $1 \le p \le \infty$. Given $\varepsilon>0$, any  $a \in \mathcal{S}_{p}$ can be factored as $a = k_1a_1k_2$ with $k_1, k_2 \in \mathcal{S}_{\infty}$, $a_1 \in \mathcal{S}_{p}$ and $\n{k_1}_{\mathcal{S}_{\infty}} \n{a_1}_{\mathcal{S}_{p}} \n{k_2}_{\mathcal{S}_{\infty}} \le (1+\varepsilon)\n{a}_{\mathcal{S}_{p}}$.
Similar one-sided factorizations can also be achieved.
\end{remark}

Therefore, any $a,b \in \mathcal{S}_{2p}$ can be factored as $a = k_3a_1k_1$, $b = k_2b_1k_4$ with $k_1,k_2,k_3,k_4 \in \mathcal{S}_\infty$, and $a_1,b_1 \in \mathcal{S}_{2p}$. This implies that the prototypical operator $p$-compact mapping (from Example \ref{example-multiplication-is-operator-compact}) $M(a,b) : \mathcal{S}_{p'} \to \mathcal{S}_1$ factors as
$$
\xymatrix{
\mathcal{S}_{p'} \ar[r]^{M(a,b)} \ar[d]_{M(k_1,k_2)} &\mathcal{S}_1 \\
\mathcal{S}_{p'} \ar[r]_{M(a_1,b_1)} &\mathcal{S}_1 \ar[u]_{M(k_3,k_4)}
}
$$
where the first mapping is completely compact, the second one is operator $p$-compact, and the third one is operator compact. While general operator $p$-compact mappings do not necessarily factor through one of these prototypical examples, we will  show below that this type of factorization can nevertheless always
be achieved. This is a noncommutative version of \cite[Prop. 2.9]{Galicer-Lassalle-Turco}, which in turn generalized \cite[Thm. 3.1]{Choi-Kim}. The presentation of the argument looks rather different, but conceptually it is very similar. We first isolate a lemma that will be used several times in the proof.

\begin{lemma}\label{lemma-factorization-linear-algebra}
Let $V_1,V_2,W$ be operator spaces, $T_j \in \CB(V_j,W)$ for $j=1,2$ and $T \in \CB(V_1,V_2)$ such that $T_1 = T_2T$. For $j=1,2$ let $Q_j : V_j \to V_j/\ker(T_j)$ be the canonical quotient and $\widetilde{T_j} : V_j/\ker(T_j) \to W$ be the associated monomorphism such that $T_j = \widetilde{T_j}Q_j$.
Then there exists $A \in \CB(V_1/\ker(T_1), V_2/\ker(T_2))$ such that $A Q_1 = Q_2 T$, that is, making the following diagram commutative,
\[
\xymatrix{
& & W &\\ 
V_1/\ker(T_1) \ar@{.>}@/_2.5pc/[rrrr]^{A} \ar@/^1.5pc/[urr]^{\widetilde{T_1}}  &V_1 \ar[ur]^{T_1} \ar@{->>}[l]_>>>>>{Q_1} \ar[rr]^T & &V_2\ar[ul]_{T_2} \ar@{->>}[r]^>>>>>{Q_2} &V_2/\ker(T_2) \ar@/_1.5pc/[ull]_{\widetilde{T_2}}\\
}
\]
and moreover $\n{A}_{\cb} \le \n{T}_{\cb}$.
If additionally $T$ is operator $p$-compact, then so is $A$ and $\kappa_p^o(A) \le \kappa_p^o(T)$.
If $T$ is completely compact, so is $A$.
\end{lemma}

\begin{proof}
 To see that  $A$ is well defined note that if $v\in\ker(T_1) $ then $Tv\in\ker(T_2)$ because $T_1v=T_2T(v)$.
Since for every $\varepsilon>0$ we have
\[
A(\mathbf{B}_{V_1/\ker(T_1)}) \subseteq AQ_1(1+\varepsilon)\mathbf{B}_{V_1} = (1+\varepsilon) Q_2T(\mathbf{B}_{V_1})
\]
the rest of the conclusions are immediate.
\end{proof}

We will now prove the aforementioned factorization result. Note that for this decomposition we need to manage both notions of compactness.

\begin{theorem}\label{thm-factorization-Choi-Kim}
Any operator $p$-compact mapping $T: V \to W$ can be factored as $AT_0B$ where $A$ is operator compact, $T_0$ is operator $p$-compact, and $B$ is completely compact. Moreover, $\kappa_p^o(T) = \inf\{ \n{A}_{\cb} \kappa_p^o(T_0) \n{B}_{\cb} \}$ where the infimum is taken over all such factorizations.
In the case $p=\infty$, we can moreover take $B$ to be operator compact. 
\end{theorem}

\begin{proof}
Let $T : V \to W$ be operator $p$-compact. To simplify the presentation, we will show separately that one can achieve factorizations of the forms $T = T_0B$ and $T=AT_0$.

Let $\varepsilon>0$ be given. Then, there exists $w \in \mathcal{S}_p[W]$ such that $T(\mathbf{B}_V) \subseteq \Theta^w(\mathbf{B}_{\mathcal{S}_p'})$ and $\n{w}_{\mathcal{S}_p[W]} \le (1+\varepsilon)\kappa_p^o(T)$.
Using the equivalent definition of operator $p$-compact mapping in terms of commutative diagrams from \cite[Def. 3.2]{ChaDiGa-Operator-p-compact}, specifically the version of the diagram obtained in 
the proof of \cite[Thm. 3.11]{ChaDiGa-Operator-p-compact} (note that the argument is analogous to that of Theorem \ref{thm-factorization-weakly-p-compact} below), there exists $B_0 \in \CB(V,\mathcal{S}_p')$ with $\n{B_0}_{\cb} \le 1$ such that the following diagram commutes
\[
\xymatrix{
V \ar[r]^{T} \ar[dr]_{B_0} &W & \mathcal{S}_p' \ar[l]_{\Theta^w} \ar@{->>}[ld]^Q\\
& \mathcal{S}_{p}'/\ker(\Theta^w) \ar[u]_{\widetilde{\Theta^w}} &
}
\]
where $Q : \mathcal{S}_p' \to \mathcal{S}_p'/\ker(\Theta^w)$ is the quotient map and $\widetilde{\Theta^w} : \mathcal{S}_{p}'/\ker(\Theta^w) \to W$ is the natural monomorphism associated to $\Theta^w$.
By \cite[Thm. 1.5]{Pisier-Asterisque} we can write $w = \alpha \cdot v \cdot \beta$ with $\alpha,\beta \in \mathcal{S}_{2p}$ and $v \in \mathcal{S}_\infty[W]$.
By Remark \ref{remark-trick}, we can write $\alpha= a \alpha_0$, $\beta = \beta_0 b$ with $\alpha_0,\beta_0 \in \mathcal{S}_{2p}$ and $a,b\in \mathcal{S}_\infty$. Letting $w_0 = \alpha_0 \cdot v \cdot \beta_0 \in \mathcal{S}_p[W]$ (here we are using \cite[Thm. 1.5]{Pisier-Asterisque} again) we have $w = a \cdot w_0 \cdot b$ with $a,b\in \mathcal{S}_\infty$, $w_0 \in \mathcal{S}_p[W]$.
Moreover, note that we can additionally assume $\n{a}_{\mathcal{S}_\infty} = \n{b}_{\mathcal{S}_\infty} = 1$ and $\n{w_0}_{\mathcal{S}_p[W]} \le (1+\varepsilon)\n{w}_{\mathcal{S}_p[W]}$.

It now follows from Lemma \ref{lemma-factorization-Pisier} that $\Theta^w = \Theta^{w_0} M(a^t,b^t)$ with $M(a^t,b^t) : \mathcal{S}_p' \to \mathcal{S}_p'$. Now let $Q_0 : \mathcal{S}_p' \to \mathcal{S}_p'/\ker(\Theta^{w_0})$ be the quotient map and $\widetilde{\Theta^{w_0}} : \mathcal{S}_{p}'/\ker(\Theta^{w_0}) \to W$ be the natural monomorphism associated to $\Theta^{w_0}$.
By Lemma \ref{lemma-factorization-linear-algebra}, there exists $B_1 \in\CB(\mathcal{S}_p'/\ker(\Theta^w), \mathcal{S}_p'/\ker(\Theta^{w_0}))$ such that $B_1 Q = Q_0M(a^t,b^t)$, with $\n{B_1}_{\cb} = \n{Q_0M(a^t,b^t)}_{\cb} \le 1$.
Furthermore, since $M(a^t,b^t)$ is completely compact (see Proposition \ref{proposition-multiplication-is-completely-compact}) so is $B_1$.
Note that
$T =  \widetilde{\Theta^w} B_0 = \widetilde{\Theta^{w_0}}B_1B_0 = T_0 B$
where $T_0 =\widetilde{\Theta^{w_0}}$ is operator $p$-compact with $\kappa_p^o(T_0) \le (1+\varepsilon)^2\kappa^o_p(T)$ and $B = B_1B_0$ is completely compact with $\n{B}_{\cb} \le 1$.
In the case $p=\infty$, note that the map $M(a^t,b^t) : \mathcal{S}_1 \to \mathcal{S}_1$ above is in fact operator compact (see Example \ref{example-multiplication-is-operator-compact}) and therefore so are $B_1$ and $B$.

For the other factorization, once again given $\varepsilon>0$ take $w \in \mathcal{S}_p[W]$ as above. From \cite[Thm. 1.5]{Pisier-Asterisque} and Remark \ref{remark-trick}, we can write $w = \alpha a_0\cdot v_0 \cdot b_0\beta$ with $\alpha,\beta \in \mathcal{S}_{2p}$, $a_0,b_0 \in \mathcal{S}_\infty$, $v_0\in \mathcal{S}_\infty[V]$,  $\n{\alpha}_{\mathcal{S}_{2p}} = \n{\beta}_{\mathcal{S}_{2p}} = \n{a_0}_{\mathcal{S}_\infty} = \n{b_0}_{\mathcal{S}_\infty} = 1$, $\n{v_0}_{\mathcal{S}_\infty[V]} \le (1+\varepsilon) \n{w}_{\mathcal{S}_p[W]}$. Using Lemma \ref{lemma-factorization-Pisier} twice, $\Theta^w = \Theta^{v_0}M(a_0^t,b_0^t)M(\alpha^t,\beta^t)$ where $M(\alpha^t,\beta^t): \mathcal{S}_p' \to \mathcal{S}_1$, $M(a_0^t,b_0^t): \mathcal{S}_1 \to \mathcal{S}_1$ and $\Theta^{v_0} : \mathcal{S}_1 \to W$.
By Lemma \ref{lemma-factorization-linear-algebra}, we can find completely bounded linear maps 
$A_1, A_2$ making the following diagram commutative (where the double-headed arrows are the canonical quotient maps and the tilde is used to denote the associated natural monomorphisms as before):
\[
\xymatrix{
 \mathcal S'_p /\ker(\Theta^w)\ar@/^1.5pc/[drr]^{\widetilde{\Theta^{w}}} \ar[ddd]_{A_1} & & &\\
&\mathcal{S}_p' \ar@{->>}[ul] \ar[r]^{\Theta^w} \ar[d]_{M(\alpha^t,\beta^t)} &W & \\
& \mathcal{S}_1 \ar@{->>}[dl] \ar[r]_{ M(a_0^t,b_0^t) } & \mathcal{S}_1 \ar[u]^{\Theta^{v_0}}\ar@{->>}[r] & \mathcal{S}_1/\ker(\Theta^{v_0}) \ar[lu]_{\widetilde{\Theta^{v_0}}} \\
\mathcal{S}_1/\ker(\Theta^{v_0}M(a_0^t,b_0^t)) \ar@/_1pc/[urrr]_{A_2}  & & &\\
}
\]
Moreover Lemma \ref{lemma-factorization-linear-algebra} additionally yields that since $M(\alpha^t,\beta^t)$ is operator $p$-compact so is $A_1$, and since $M(a_0^t,b_0^t)$ is operator compact so is $A_2$.
Thus we have obtained a factorization $T = \widetilde{\Theta^{v_0}}A_2A_1B_0$  where $A_1B_0$ is operator $p$-compact with
$\kappa_p^o(A_1B_0) \le \kappa^o_p(A_1) \n{B_0}_{\cb} \le \n{\alpha^t}_{\mathcal{S}_{2p}} \n{\beta^t}_{\mathcal{S}_{2p}} \le 1$, and $\widetilde{\Theta^{v_0}}A_2$ is operator compact with $\n{\widetilde{\Theta^{v_0}}A_2}_{\cb} \le \n{v_0}_{\mathcal{S}_\infty[W]}\n{a_0^t}_{\mathcal{S}_\infty}\n{b_0^t}_{\mathcal{S}_\infty} \le (1+\varepsilon)^2\kappa^o_p(T)$. 
\end{proof}

\begin{remark}
Theorem \ref{thm-factorization-Choi-Kim} is also valid with $\kappa_\infty^o(A)$ in place of $\n{A}_{\cb}$.
\end{remark}

\section{Adjoints of operator $p$-compact maps}\label{sec-adjoints}

Both in the classical and in the noncommutative setting it is a relevant issue, given a mapping ideal, to have a description of their elements through their adjoints mappings.
In the Banach space realm, a linear map $T : X \to Y$ is $p$-compact if and only if its adjoint $T' : Y' \to X'$ is quasi $p$-nuclear \cite[Prop. 3.8]{Delgado-Pineiro-Serrano-adjoints}, see also \cite[Cor. 2.7]{Galicer-Lassalle-Turco} for the isometric version (the latter notion will not play a role in the present paper so we will not define it precisely, but the corresponding operator space notion is presented below in Definition \ref{defn-quasi-completely-p-nuclear}).

In the operator space framework, no fully general characterization of any of the various notions of compactness for a mapping is known in terms of a property of its adjoint.
The only available results are limited to maps with finite-dimensional codomains: \cite[Thm. 4.3]{Yew} shows that when $W$ is finite-dimensional, $T : V \to W$ is $Z$-compact if and only if its adjoint $T': W'\to V'$ factors through a subspace of $Z$ (and with equality of norms). This is a generalization of \cite[Cor. 4.15]{Oikhberg-Thesis}, which covers the case of operator compact mappings.

In this section, inspired by analogous work in the Banach space case \cite{Delgado-Pineiro-Serrano-adjoints}, we obtain similar and more general results for operator $p$-compact mappings.
The adjoint of an operator $p$-compact mapping is quasi $p$-nuclear in the appropriate sense  (see below for the definition and Proposition \ref{prop:adjoint-of-p-compact} for the result).
The reverse implication holds when the codomain has finite dimension, just as in the aforementioned results of \cite{Oikhberg-Thesis,Yew}, and more generally whenever the codomain is completely contractively complemented in its bidual (see Proposition \ref{prop:adjoint-of-p-compact-finite-dimension}).

To understand the behavior of the adjoint of a $p$-compact mapping, we will need two definitions related to the so-called nuclearity properties.

\begin{definition}\label{defn-completely-p-nuclear}
Let $V$ and $W$ be operator spaces and $T : V \to W$ be a linear map.
Following \cite[Def. 3.1.3.1]{Junge-Habilitationschrift},
we say that $T : V \to W$ is \emph{completely $p$-nuclear} if there exist $a,b \in \mathcal{S}_{2p}$ such that $T$  admits a factorization
$$
\xymatrix{
V \ar[r]^T \ar[d]_{R} &W \\
\mathcal{S}_{\infty} \ar[r]_{M(a,b)} &\mathcal{S}_p \ar[u]_{S}
}
$$
where $R,S$ are completely bounded linear maps.
The completely $p$-nuclear norm of $T$ is defined as
$
\nu_p^{\circ}(T) = \inf \big\{ \n{R}_{\cb} \n{S}_{\cb}  \n{a}_{\mathcal{S}_{2p}} \n{b}_{\mathcal{S}_{2p}}  \big\}
$
where the infimum is taken over all factorizations as above.
We denote the class of all completely $p$-nuclear mappings $T:V \to W$ by $\mathcal N_p^o(V,W)$,  and this is a mapping ideal. For simplicity, if $p=1$ we write $\nu_1^o$ as $\nu^o$. The ideal in this case corresponds to the well-known class of completely nuclear mappings \cite[Sec. 12.2]{Effros-Ruan-book}.
\end{definition}

Note that in the factorization in Definition \ref{defn-completely-p-nuclear} above $\mathcal{S}_\infty$ can be replaced by $\mathcal{B}(\ell_2)$, because of the following result 
contained in the proof of \cite[Lem. 6.4]{CD-completely-mixing} (and which follows easily from Remark \ref{remark-trick}).

\begin{lemma}\label{lemma:multiplication-operator-on-B(ell_2)}
Let $1 \le p \le \infty$, and let $a,b\in \mathcal{S}_{2p}$. For every $\varepsilon>0$, the multiplication mapping $M=M(a,b) : \mathcal{B}(\ell_2) \to \mathcal{S}_p$ admits a factorization $M = M_1 \circ M_2$ where $M_2 : \mathcal{B}(\ell_2) \to \mathcal{S}_\infty$ is a complete contraction and  $M_1= M(a_1,b_1) : \mathcal{S}_\infty \to \mathcal{S}_p$ is a multiplication mapping with $a_1 , b_1 \in \mathcal{S}_{2p}$ satisfying $\n{a_1}_{\mathcal{S}_{2p}} \le (1+\varepsilon)\n{a}_{\mathcal{S}_{2p}}$ and $\n{b_1}_{\mathcal{S}_{2p}} \le (1+\varepsilon)\n{b}_{\mathcal{S}_{2p}}$.
\end{lemma}

The second nuclearity notion was introduced in \cite{ChaDiGa-Operator-p-compact}:

\begin{definition}\label{defn-comp-right-p-nuclear}
Let $V$ and $W$ be operator spaces and $T : V \to W$ a linear map.
We say that $T : V \to W$ is \emph{completely right $p$-nuclear} if there exist $a,b \in \mathcal{S}_{2p}$ such that $T$  admits a factorization
$$
\xymatrix{
V \ar[r]^T \ar[d]_{R} &W \\
\mathcal{S}_{p'} \ar[r]_{M(a,b)} &\mathcal{S}_1 \ar[u]_{S}
}
$$
where $R,S$ are completely bounded linear maps.
The completely right $p$-nuclear norm of $T$ is defined as
$
\nu^p_o(T) = \inf \big\{ \n{R}_{\cb} \n{S}_{\cb}  \n{a}_{\mathcal{S}_{2p}} \n{b}_{\mathcal{S}_{2p}}  \big\}
$
where the infimum is taken over all factorizations as above.
We denote the class of all completely right $p$-nuclear mappings $T:V \to W$ by $\mathcal N^p_o(V,W)$,  and this is a mapping ideal.
\end{definition}

The reader is cautioned to be mindful of the similarity in the notations for completely $p$-nuclear and completely right $p$-nuclear mappings, where the only difference is whether the $p$ is a subindex $(\mathcal N_p^o; \nu_p^o)$ or a superindex $(\mathcal N^p_o; \nu^p_o)$. Similar notations have traditionally been used in the Banach space literature, so we have chosen to be consistent with that.

\begin{remark}
Since our interest in this work is the operator $p$-compact mappings, we touch on the completely $p$-nuclear and completely right $p$-nuclear ones only very superficially, without intending to address their own theory. For example, it is very natural to wonder about the alternative version of Definition \ref{defn-comp-right-p-nuclear} where the multiplication maps $M(a,b) : \mathcal{S}_{p'} \to \mathcal{S}_1$ are replaced by their finite-dimensional versions $\mathcal{S}^n_{p'} \to \mathcal{S}^n_1$. 
In Lemma \ref{lemma-finite-right-p-nuclear} below we check a case where those two notions agree, 
and we thank an anonymous referee for reminding us to clarify that a far more detailed study of this and other related questions was already done by Junge in \cite[Sec. 3.1]{Junge-Habilitationschrift}.
For example, our aforementioned Lemma \ref{lemma-finite-right-p-nuclear} is essentially a special case of \cite[Cor. 3.1.4.5]{Junge-Habilitationschrift}.

The main difference between Junge's work and our own is the choice of primary object: for us the fundamental objects are defined using the infinite-dimensional multiplication maps $M(a,b) : \mathcal{S}_{p'} \to \mathcal{S}_1$, whereas for Junge the focus is on the finite-dimensional ones (see the definition of the $\gamma_{pq,\mathcal{M}_{pq}}$ mappings in \cite[Sec. 3.1.4]{Junge-Habilitationschrift}). This is due to a difference in goals, and means that our results are closely related to Junge's but difficult to compare directly. The similarity is not at all accidental: \cite{Junge-Habilitationschrift} was a crucial inspiration for us, and we use its results and techniques repeatedly throughout this article.
\end{remark}

Another notion that has its roots in the theory of Banach space mappings related with $p$-compactness is defined as follows:

\begin{definition}\label{defn-quasi-completely-p-nuclear}
Let $V$ and $W$ be operator spaces and $T : V \to W$ a linear map.
We say that $T$ is \emph{quasi completely $p$-nuclear} if $j \circ T : V \to Y$ is completely $p$-nuclear, where $j : W \to Y$ is a completely isometric embedding of $W$ into an injective operator space $Y$. We will denote $q\nu^o_p(T) = \nu^o_p(j \circ T)$.
Note that this definition is independent of the particular embedding.
We denote the class of all quasi completely $p$-nuclear mappings $T:V \to W$ by $\mathcal{QN}^o_p(V,W)$,  and this is a mapping ideal.
\end{definition}

The upcoming proposition's proof requires a specific construction. For an operator space $V$, there exists a set $I$ and a family $(n_i)_{i \in I} \subset \mathbb{N}$ such that $V$ can be represented as the quotient of  $\ell_1(\{\mathcal{S}_1^{n_i} \}_{i \in I})$, as detailed in \cite[Prop. 2.12.2]{Pisier-OS-theory}. We denote this space by $Z_V$, and $Q_V : Z_V \rightarrow V$ represents the corresponding complete 1-quotient mapping. Notably, $Z_V$ is projective, implying that its dual space $Z_V'$ is injective \cite[Chapter 24]{Pisier-OS-theory}.

The subsequent result is an operator space version of \cite[Cor. 3.4 and Prop. 3.8]{Delgado-Pineiro-Serrano-adjoints}, with the significant difference that in part $(a)$ we only get one implication and not the equivalence.

\begin{proposition}\label{prop:adjoint-of-p-compact}
Let $1 \le p \le \infty$.
\begin{enumerate}[(a)]
\item If $T \in \mathcal{K}^o_p(V,W)$, then $T' \in \mathcal{QN}^o_p(W',V')$ and moreover $q\nu^o_p(T') \le \kappa^o_p(T)$.
\item $T \in \mathcal{QN}^o_p(V,W)$ if and only if $T' \in \mathcal{K}^o_p(W',V')$ and moreover $\kappa^o_p(T') = q\nu_p^o(T)$.
\item $T \in \mathcal{QN}^o_p(V,W)$ if and only if $T'' \in \mathcal{QN}^o_p(V'',W'')$, and moreover $q\nu_p^o(T) = q\nu_p^o(T'')$.
\end{enumerate}

\end{proposition}

\begin{proof}
$(a)$ Assume that $T \in \mathcal{K}^o_p(V,W)$. By \cite[Prop. 3.8 and Thm. 2.8]{ChaDiGa-Operator-p-compact}
there is a commutative diagram for $TQ_V$
\[
\xymatrix{
Z_V \ar[r]^{Q_V} \ar[d] & V \ar[r]^T &W\\
\mathcal{S}_{p'} \ar[r]_{M(a,b)} &\mathcal{S}_1 \ar[ru] &
}
\]
where $a,b\in \mathcal{S}_{2p}$ and $\kappa^o_p(T)$ is the infimum of the product of the $\cb$-norms of the mappings involved in this factorization.
Dualizing gives
\[
\xymatrix{
W' \ar[r]^{T'} \ar[d] & V' \ar[r]^{Q_V'} &Z_V'\\
\mathcal{B}(\ell_2) \ar[r]_{M(b,a)} &\mathcal{S}_{p'}' \ar[ru] &
}
\]
Note that for $1 \le p < \infty$, $\mathcal{S}_{p'}'= \mathcal{S}_p$.
In the case $p=\infty$, $\mathcal{S}_{\infty'}'= \mathcal{B}(\ell_2)$ so the lower row of the above diagram is $M(b,a) : \mathcal{B}(\ell_2) \to \mathcal{B}(\ell_2)$. But since $a,b \in \mathcal{S}_\infty$, the multiplication map $M(b,a)$ actually takes values in $\mathcal{S}_\infty$, so we conclude that for all $1\le p \le \infty$ we have a diagram 
\[
\xymatrix{
W' \ar[r]^{T'} \ar[d] & V' \ar[r]^{Q_V'} &Z_V'\\
\mathcal{B}(\ell_2) \ar[r]_{M(b,a)} &\mathcal{S}_p \ar[ru] &
}
\]
showing that $Q_V'T'$ is completely $p$-nuclear by Lemma \ref{lemma:multiplication-operator-on-B(ell_2)}.
Since $Z_V'$ is an injective operator space and $Q_V'$ is a completely isometric injection, we conclude that $T' \in \mathcal{QN}^o_p(W',V')$. Taking the infimum over all such factorizations of $T$ yields $q\nu^o_p(T') \le \kappa^o_p(T)$.

 $(b)$ (Forward implication) Let us assume that $T \in \mathcal{QN}^o_p(V,W)$.
Consider the canonical complete quotient $Q_{W'} : Z_{W'} \to W'$, whose adjoint $Q_{W'}' : W'' \to Z_{W'}'$ is a completely isometric injection. Since $Z_{W'}'$ is an injective operator space, from the assumption we have a factorization
\[
\xymatrix{
V \ar[d] \ar[r]^T &W \ar[r]^{\iota_W} &W'' \ar[r]^{Q'_{W'}} &Z_{W'}'\\
\mathcal{S}_{\infty} \ar[rr]_{M(a,b)} & &\mathcal{S}_p \ar[ru] & 
}
\]
where $a,b\in \mathcal{S}_{2p}$.
Dualizing, we have
\[
\xymatrix{
Z_{W'}'' \ar[d] \ar[r]^{Q_{W'}''} &W''' \ar[r]^{\iota_W'} &W' \ar[r]^{T'}& V' \\
\mathcal S'_{p} \ar[rr]_{M(b,a)} & &\mathcal{S}_1 \ar[ru] & 
}
\]
Restricting to $Z_{W'} \subset Z_{W'}''$ we have 
\[
\xymatrix{
Z_{W'} \ar[d] \ar[r]^{Q_{W'}} &W' \ar[r]^{T'}& V' \\
\mathcal S'_{p} \ar[r]_{M(b,a)} & \mathcal{S}_1 \ar[ru] & 
}
\]
Note that in the above diagram we can replace $\mathcal{S}_p'$ by $\mathcal{S}_{p'}$: these spaces are equal when $1 < p \le \infty$, and for $p=1$ it follows from Lemma \ref{lemma:multiplication-operator-on-B(ell_2)}.
Therefore, $T' \in \mathcal{K}^o_p(W',V')$ by appealing to \cite[Prop. 3.8]{ChaDiGa-Operator-p-compact} again.
Once more, taking the infimum over all such factorizations for $Q_{W'}'{\iota_W}T$ yields $\kappa^o_p(T') \le q\nu_p^o(T)$.

$(c)$ If $T \in \mathcal{QN}^o_p(V,W)$, by the previous two parts we conclude $T'' \in \mathcal{QN}^o_p(V'',W'')$ and $q\nu_p^o(T'') \le q\nu_p^o(T)$.
Now suppose that $T'' \in \mathcal{QN}^o_p(V'',W'')$, and consider $j : W'' \to Z$ a completely isometric embedding of $W''$ into an injective operator space $Z$. Then $j \circ T''$ is completely $p$-nuclear, and hence so is $j \circ T'' \circ \iota_V = j \circ \iota_W \circ T$, which shows that $T \in \mathcal{QN}^o_p(V,W)$ since $j \circ \iota_W$ is a completely isometric embedding of $W$ into the injective operator space $Z$. Moreover, using the ideal property of $\nu_p^o(\cdot)$, 
\[
q\nu_p^o(T) =  \nu_p^o(j \circ \iota_W \circ T) = \nu_p^o(j \circ T'' \circ \iota_V) \le \nu_p^o(j \circ T'' ) \n{\iota_V}_{\cb} = q\nu_p^o(T'').
\]

$(b)$ (Reverse implication) If $T' \in \mathcal{K}^o_p(W',V')$, applying the previous parts we get $T'' \in \mathcal{QN}^o_p(V'',W'')$ and therefore $T \in \mathcal{QN}^o_p(V,W)$, with $q\nu_p^o(T) = q\nu_p^o(T'') \le \kappa_p^o(T')$.
\end{proof}

In the particular case where the codomain is complemented in its bidual we have the converse of Proposition \ref{prop:adjoint-of-p-compact} $(a)$.

\begin{proposition}\label{prop:adjoint-of-p-compact-finite-dimension}
Let $1 \le p \le \infty$.
Let $T : V \to W$ be a linear map.
Suppose that $W$ is completely contractively complemented in its bidual (in particular, if $W$ has finite dimension or is a dual space).
Then $T \in \mathcal{K}^o_p(V,W)$ if and only if $T' \in \mathcal{QN}^o_p(W',V')$, and moreover $\kappa^o_p(T) = q\nu^o_p(T')$.
\end{proposition}

\begin{proof}
The ``only if'' implication is Proposition \ref{prop:adjoint-of-p-compact} $(a)$. Suppose now that  $T' \in \mathcal{QN}^o_p(W',V')$. By Proposition \ref{prop:adjoint-of-p-compact} $(b)$, $T'' \in \mathcal{K}^o_p(V'',W'')$ and $\kappa_p^o(T'') \le  q\nu^o_p(T')$.
If $P : W'' \to W$ is a completely contractive projection, then $T = PT''\iota_V$, so $T$ is operator $p$-compact and 
$\kappa_p^o(T) \le \n{P}_{\cb} \kappa_p^o(T'') \n{\iota_V}_{\cb} \le  q\nu^o_p(T')$, finishing the proof.
\end{proof}

Note that the case $p=\infty$ of Proposition \ref{prop:adjoint-of-p-compact-finite-dimension} generalizes the case of operator compact mappings with finite-dimensional codomain from \cite[Cor. 4.15]{Oikhberg-Thesis}. Specifically, it can be readily verified that for a mapping with a finite-dimensional domain, the $q\nu_\infty^o$ norm coincides with the norm of $\cb$-factorization through a subspace of $\mathcal{S}_\infty$, leading to the conclusion.

Up to this point, the only examples for which we had been able to calculate the operator $p$-compact norm are the multiplication mappings $M(a,b) : \mathcal{S}_{p'} \to \mathcal{S}_1$ with $a,b \in \mathcal{S}_{2p}$ (Example \ref{example-multiplication-is-operator-compact}), all of which take values in $\mathcal{S}_1$. Now, Proposition \ref{prop:adjoint-of-p-compact-finite-dimension} allows us to also consider multiplication mappings defined from $\mathcal{S}_1$ to $\mathcal{S}_q$, for certain values of $q$.

\begin{corollary}\label{cor-examples-p-compact}
Let $1\le p, q, r \le \infty$, $a,b \in \mathcal{S}_{2q'}$. In each of the following situations, we have ${\kappa_p^o(M(a,b) : \mathcal{S}_1 \to \mathcal{S}_q)} = \n{a}_{\mathcal{S}_{2r}} \n{b}_{\mathcal{S}_{2r}}$:
\begin{enumerate}[(a)]
\item $q \le p$, $\frac{1}{r} = \frac{1}{q'}+\frac2p$.
\item $\max\{2,p\} \le q$, $\frac{1}{r} = \frac1q+1$. 
\end{enumerate}
In particular, under the previous conditions, if either $a$ or $b$ belongs to $\mathcal{S}_{2q'} \setminus \mathcal{S}_{2r}$ and neither of them is zero, then $M(a,b) : \mathcal{S}_1 \to \mathcal{S}_q$ is not operator $p$-compact.
\end{corollary}

\begin{proof}
It will suffice to check the finite-dimensional case, and the general one will follow by approximation.
From Proposition \ref{prop:adjoint-of-p-compact-finite-dimension},
\[
\kappa_p^o(M(a,b) : \mathcal{S}_1^n \to \mathcal{S}_q^n) = q\nu^o_p(M(b,a): \mathcal{S}_{q'}^n \to M_n) =  \nu^o_p(M(b,a): \mathcal{S}_{q'}^n \to M_n), 
\]
so the conclusion now follows from \cite[Thm. 3.1]{Oikhberg}.
\end{proof}

Finally, we show that the monotonicity relation given by Proposition \ref{prop-monotonicity-compactness} is strict.

\begin{corollary} \label{cor: p-not-q}
Let $1 \le q < p \le \infty$. Then there exists a linear mapping which is operator $p$-compact but not operator $q$-compact (and therefore, there exists a relatively operator $p$-compact matrix set which is not relatively operator $q$-compact).
\end{corollary}

\begin{proof}
Let $\frac{1}{r_q} = 1+\frac1q$ and $\frac{1}{r_p} = \frac{1}{q'}+\frac{2}{p}$.
Since $q<p$, an easy calculation shows $r_p>r_q$.
Let $a,b \in \mathcal{S}_{2r_p} \setminus \mathcal{S}_{2r_q}$, and consider $M(a,b):\mathcal{S}_1 \to \mathcal{S}_q$. By Corollary \ref{cor-examples-p-compact} this is an operator $p$-compact mapping, but not operator $q$-compact.
\end{proof}

\section{Operator weakly $p$-compact matrix sets and mappings}\label{sect-weakly-p-compact}

In the classical setting, just as there exists the notion of $p$-compactness for sets and mappings, there is also a weaker notion called weak $p$-compactness (see e.g., \cite[Def. 2.3 and 2.4]{Sinha-Karn}). We delve into their analogous version within our framework.
We first need to recall the non-commutative counterpart of the sequence space $\ell_p^w(X)$ which was introduced in \cite{ChaDiGa-Operator-p-compact}:  
$$ S_p^w[V] := \left\{ v=(v_{ij})_{i, j =1}^{\infty}  : \forall i,j, \; v_{ij}\in V \; \mbox{and} \;  \sup_N \left\| (v_{ij})_{i,j=1}^N \right\|_{S_p^N \widehat{\otimes}_{\min} V} < \infty \right\}. $$  
Equipped with the matricial norm structure defined by

$$ \Vert  \big ((v_{ij}^{kl})_{i,j} \big)_{k,l=1}^n  \Vert_{M_n(S_{p}^w[V])} := \sup_{N} \Vert \big((v_{ij}^{kl})_{i,j=1}^N \big)_{k,l=1}^{n} \Vert_{M_n(S_p^N \widehat{\otimes}_{\min} V) },$$
this defines an operator space. As shown in \cite[Lem. 2.4]{ChaDiGa-Operator-p-compact}, the space $S_p^w[V]$ is completely isometric to $\CB(\mathcal{S}_{p'},V)$ under the identification $v \mapsto \Theta^{v}$.

\begin{definition}
Let $1 \le p \le \infty$.
 A matrix set $\mathbf{K}=(K_n)$ over $V$ is called \emph{relatively operator weakly $p$-compact} if there exists $v \in \mathcal{S}_p^w[V]$ such that $\mathbf{K}\subset \Theta^v(\mathbf{B}_{\mathcal{S}_p'})$. In that case  we define
\[
\frak{ m}^{w,o}_p(\mathbf{K}) = \inf\left\{\| v\|_{ \mathcal{S}_p^w[V] } \;:\; v\in  \mathcal{S}_p^w[V], \; \mathbf{K}\subset \Theta^v(\mathbf{B}_{\mathcal{S}_p'}) \right\}.
\]
A linear map $T : V \to W$ is called \emph{operator weakly $p$-compact} if the matrix set $T(\mathbf{B}_V)$ is relatively operator weakly $p$-compact, and in this case we define the operator weakly $p$-compact norm of $T$ by
\[
\omega_p^o(T) = \frak{ m}^{w,o}_p\left( T(\mathbf{B}_V) \right)
\]
We denote by $\mathcal{W}_p^o(V,W)$ the set of all operator weakly $p$-compact maps from $V$ to $W$.
\end{definition}

We emphasize that, in contrast with all the other classes of mappings considered in this paper, $\mathcal{W}_p^o$ is not known to be a mapping ideal since we have not been able to identify a natural operator space structure for it. Nevertheless, it is at least normed and satisfies the ideal property (see Proposition \ref{prop-W_p^o-ideal-property} below).

It is immediate from the definitions that every operator $p$-compact mapping is operator weakly $p$-compact, and moreover the inclusion $\mathcal{K}_p^o(V,W) \to \mathcal{W}_p^o(V,W)$ is contractive. 

In analogy with the comment after the proof of Lemma \ref{lemma-Theta-well-defined}, the relatively weakly operator $p$-compactness of $\mathbf{K}$ can be alternatively expressed as $\mathbf{K}$ being a subset of $\overline{\mathbf{co}_p(v)}$ for some $v \in \mathcal{S}_p^w[V]$ with $\frak{ m}^{w,o}_p(\mathbf{K}) = \inf\left\{\| v\|_{ \mathcal{S}_p^w[V] } \;:\; v\in  \mathcal{S}_p^w[V], \; \mathbf{K}\subset \overline{\mathbf{co}_p(v)}\} \right\}.$

Since $\mathcal{S}_p^w[V] = \CB(\mathcal{S}_{p'}, V)$, we can equivalently say that a matrix set $\mathbf{K}$ over $V$  is relatively operator weakly $p$-compact if there exists a mapping $\Theta\in \CB(\mathcal{S}_{p'}, V)$ such that $\mathbf{K}\subseteq \Theta(\mathbf{B}_{\mathcal{S}_{p'}})$ and we have
\begin{equation}\label{eqn-omegap-mas-simple}
\frak{ m}^{w,o}_p(\mathbf{K}) = \inf\left\{ \n{\Theta}_{\cb} \;:\; \Theta\in\CB(\mathcal{S}_{p'},V), \; \mathbf{K}\subseteq \Theta(\mathbf{B}_{\mathcal{S}_{p'}}) \right\}.
\end{equation}
It now follows obviously that any $\Theta \in \CB(\mathcal{S}_{p'},W)$ satisfies $\omega^o_p(\Theta) \le \n{\Theta}_{\cb}$. As the reverse inequality always holds, we conclude that $\mathcal{W}_p^o(\mathcal{S}_{p'},W) = \CB(\mathcal{S}_{p'},W)$ isometrically.

Comparing with Yew's definition of $\mathcal{S}_p$-compact mappings (see Definition \ref{defn-Yew-compact}) it becomes evident that such mappings are weakly $p$-compact, given that $\mathcal{S}_p\widehat\otimes_{\min} V \subset \mathcal{S}_p^w[V]$.

In the Banach space setting, $p$-compact and weakly $p$-compact maps have very similar factorizations \cite[Thms. 3.1 and 3.2]{Sinha-Karn}.
We already have factorizations for operator $p$-compact maps, so one would expect something similar to hold for operator weakly $p$-compact maps. Theorem \ref{thm-factorization-weakly-p-compact} below provides this and should be compared to \cite[Eqn. (10)]{ChaDiGa-Operator-p-compact}.

Recall from \cite[Sec. 7.2]{Pisier-Asterisque} that $\Gamma_{\mathcal{S}_p}(V,W)$ denotes the space of completely bounded linear maps $T: V \to W$ admitting a factorization of the form $T = ba$ with $a\in\CB(V,\mathcal{S}_p)$ and $b\in\CB(\mathcal{S}_p,W)$, with the norm $\gamma_{\mathcal{S}_p}(T) = \inf \{ \n{a}_{\cb} \n{b}_{\cb}\}$ where the infimum is taken over all such factorizations. Moreover, $\Gamma_{\mathcal{S}_p}(V,W)$ is complete when $W$ is.

\begin{theorem}\label{thm-factorization-weakly-p-compact}
A linear map $T : V \to W$ is operator weakly $p$-compact if and only if there exist an operator space $G$, and mappings $\Theta \in \Gamma_{\mathcal{S}_{p'}}(G, W)$ and $R \in\CB(G,G/\ker(\Theta))$ such that the following diagram commutes
\begin{equation}\label{eqn-factorization-weakly-p-compact}
\xymatrix{
V \ar[r]^T\ar[rd]_R    & W  & \ar[l]_{\Theta } \ar@{->>}[ld]^{\pi}  G \\
 &   G/\ker\Theta  \ar[u]^{\widetilde{\Theta}},  &
}.
\end{equation}
where $\pi$ and $\widetilde{\Theta}$  stand for the natural $1$-quotient mapping and the natural monomorphism associated to $\Theta$, respectively.
Moreover, in this case $\omega_p^o(T)$ is equal to the infimum of $\n{R}_{\cb}\gamma_{\mathcal{S}_{p'}}(\Theta)$ over all such factorizations.
Additionally, one can consider only factorizations with $G=\mathcal{S}_{p'}$.
\end{theorem}

\begin{proof}
Suppose that $T : V \to W$ admits a factorization as in \eqref{eqn-factorization-weakly-p-compact}, and factor 
 $\Theta = ba$ with $a\in\CB(G,\mathcal{S}_{p'})$ and $b\in\CB(\mathcal{S}_{p'},W)$.
Then for every $\varepsilon>0$,
\begin{multline*}
T(\mathbf{B}_V) = \widetilde{\Theta} R (\mathbf{B}_V) \subseteq \n{R}_{\cb} \widetilde{\Theta}(\mathbf{B}_{G/\ker(\Theta)}) \subseteq  
\n{R}_{\cb} (1+\varepsilon) \widetilde{\Theta} \pi (\mathbf{B}_{G}) \\
=
\n{R}_{\cb} (1+\varepsilon)\Theta(\mathbf{B}_{G}) = 
\n{R}_{\cb} (1+\varepsilon)ba(\mathbf{B}_{G}) \subseteq 
\n{R}_{\cb} \n{a}_{\cb} (1+\varepsilon)b(\mathbf{B}_{\mathcal{S}_{p'}})
\end{multline*}
and therefore $T$ is weakly $p$-compact with $\omega^o_p(T) \le \n{R}_{\cb} \gamma_{\mathcal{S}_{p'}}(\Theta)$, so that $\omega_p^o(T)$ is less than or equal to the infimum in the statement.

Suppose now that $T : V \to W$ is weakly $p$-compact.
This means that $T(\mathbf{B}_V)$ is weakly $p$-compact, so by \eqref{eqn-omegap-mas-simple} there exists $\Theta \in \CB(\mathcal{S}_{p'}, W)$ such that $T(\mathbf{B}_V) \subseteq \Theta(\mathbf{B}_{\mathcal{S}_{p'}})$.
Define $R : V \to \mathcal{S}_{p'}/\ker\Theta$ by $Rv = [u]$ where $u \in \mathcal{S}_{p'}$ satisfies $\Theta u = Tv$.
First observe that such $u$ must exist because the range of $T$ is contained in the range of $\Theta$. Moreover, this is well defined: if $\Theta u_1 = \Theta u_2 = Tv$, obviously $[u_1] = [u_2]$. Linearity of $R$ is clear from the linearity of $T$ and $\Theta$.
To see that $R$ is completely bounded, note that if $v \in B_{M_n(V)}$, we can find $u \in B_{M_n(\mathcal{S}_{p'})}$ such that $T_n v = \Theta_n u$, and therefore
$R_n v = [u]$, so $\n{R_nv} \le 1$ and thus $R$ is a complete contraction.
We have thus found a factorization as in \eqref{eqn-factorization-weakly-p-compact} with $G=\mathcal{S}_{p'}$, and in this case $\n{R}_{\cb}\gamma_{\mathcal{S}_{p'}}(\Theta) \le \n{\Theta}_{\cb}$, so the infimum in the statement is in fact equal to $\omega_p^o(T)$.
\end{proof}

As expected we have the following result.

\begin{proposition}\label{prop-W_p^o-ideal-property}
$\omega^o_p$ is a  norm on  $\mathcal{W}_p^o(V,W)$ satisfying the ideal property.   
\end{proposition}

\begin{proof}
The ideal property is clear, and so is homogeneity.
The triangle inequality is proved in exactly the same way as in \cite[Prop. 3.3]{ChaDiGa-Operator-p-compact}, since the factorizations in \eqref{eqn-factorization-weakly-p-compact} have exactly the same form as those in \cite[Def. 3.2]{ChaDiGa-Operator-p-compact} with the only difference that $\mathcal{N}_o^p$ is replaced by $\Gamma_{\mathcal{S}_{p'}}$.
\end{proof}

The following two propositions have proofs analogous to those of \cite[Prop. 3.7 and 3.8]{ChaDiGa-Operator-p-compact}.

\begin{proposition}
Let $V$ be a projective operator space. Then, $T \in  \mathcal{W}_p^o(V,W)$ if and only if $T \in \Gamma_{\mathcal{S}_{p'}}(V,W)$ and $\omega_p^o(T) = \gamma_{\mathcal{S}_{p'}}(T)$. 
\end{proposition}

\begin{proposition}\label{prop-Wp-is-Gamma-p-sur}
Let $V$ and $W$ be operator spaces. Then $T \in  \mathcal{W}_p^o(V,W)$ if and only if $T  Q_V \in \Gamma_{\mathcal{S}_{p'}}(Z_V,W)$ and $\omega_p^o(T) = \gamma_{\mathcal{S}_{p'}}(T Q_V)$.    
\end{proposition}

By the completeness of $\Gamma_{\mathcal{S}_{p'}}(Z_V,W)$ and the previous proposition we easily derive:

\begin{corollary}
   Let $V$ and $W$ be operator spaces. Then $\mathcal{W}_p^o(V,W)$ is a Banach space.
\end{corollary}

\section{Relations to completely $p$-summing maps}\label{sect-completely-summing}

Another class of mappings related to the notions of $p$-compactness and weakly $p$-compactness in the classical realm is the ideal of absolutely $p$-summing mappings. We study similar interactions in our noncommutative framework.

Recall that for $1 \le p \le \infty$ a linear map $T : V \to W$ between operator spaces is called completely $p$-summing \cite[Chap. 5]{Pisier-Asterisque} if 
\[
Id_{\mathcal{S}_p}\otimes T : \mathcal{S}_p \otimes_{\min} V \to \mathcal{S}_p[W]
\]
is bounded, and the norm $\pi^o_p(T)$ of this map is the completely $p$-summing norm of $T$.
We denote by $\Pi_p^o(V,W)$ the set of all completely $p$-summing maps from $V$ to $W$, and this is a mapping ideal.
The basic example of a completely $p$-summing map is the multiplication map $M(a,b) : \mathcal{B}(\ell_2) \to \mathcal{S}_p$ where $a,b \in \mathcal{S}_{2p}$
\cite[Prop. 5.6]{Pisier-Asterisque}, which in particular implies that completely $p$-nuclear maps are completely $p$-summing and we always have $\pi_p^o \le \nu_p^o$.

We next note that in the definition of completely $p$-summing map one could equivalently use $\mathcal{S}_p^w[V]$ instead of $\mathcal{S}_p \otimes_{\min} V$.

\begin{proposition}\label{prop-p-summing-with-Sp-weak}
A linear map $T : V \to W$ between operator spaces is completely $p$-summing if and only if
\[
Id_{\mathcal{S}_p}\otimes T : \mathcal{S}_p^w[V] \to \mathcal{S}_p[W]
\]
is bounded. Moreover, the norm of this map coincides with $\pi_p^o(T)$.
\end{proposition}

\begin{proof}
Since $\mathcal{S}_p \otimes_{\min} V$ is completely isometric to a subspace of $\mathcal{S}_p^w[V]$, it is clear that when ${Id_{\mathcal{S}_p}\otimes T : \mathcal{S}_p^w[V] \to \mathcal{S}_p[W]}$ is bounded we have that $T$ is completely $p$-summing and moreover $\pi_p^o(T) \le \n{Id_{\mathcal{S}_p}\otimes T : \mathcal{S}_p^w[V] \to \mathcal{S}_p[W]}$.

Suppose now that $T$ is completely $p$-summing.
Restricting to subspaces, note that for every $N\in\N$ we have $\n{Id_{\mathcal{S}_p^N}\otimes T : \mathcal{S}_p^N \otimes_{\min} V \to \mathcal{S}_p^N[W]} \le \pi_p^o(T)$.
If $v=(v_{ij})_{i,j=1}^\infty \in \mathcal{S}_p^w[V]$, by the above we have that for each $N\in\N$,
\[
\n{(Tv_{ij})_{i,j=1}^N}_{\mathcal{S}_p^N[W]} \le \pi_p^o(T) \n{(v_{ij})_{i,j=1}^N}_{\mathcal{S}_p^N \otimes_{\min} V}
\]
Taking the supremum over $N$, it follows from \cite[Lem. 1.12]{Pisier-Asterisque} that $(Id_{\mathcal{S}_p} \otimes T)(v) \in \mathcal{S}_p[W]$ and moreover
$\n{(Id_{\mathcal{S}_p} \otimes T)(v)}_{\mathcal{S}_p[W]} \le \pi_p^o(T) \n{v}_{\mathcal{S}_p^w[V]}$, so
$\n{Id_{\mathcal{S}_p}\otimes T : \mathcal{S}_p^w[V] \to \mathcal{S}_p[W]} \le \pi_p^o(T)$
which completes the proof.
\end{proof}

With the previous result at hand, we can now derive a statement analogous to \cite[Prop. 5.4]{Sinha-Karn} that captures the interaction we had previously mentioned.

\begin{proposition}\label{prop-p-summing-maps-weakly-p-compact-to-p-compact}
For all $1\le p \le \infty$, completely $p$-summing maps send relatively weakly $p$-compact matrix sets to relatively operator $p$-compact matrix sets.
Furthermore, if $T \in \mathcal{W}_p^o(Z,V)$ and $A \in \Pi_p^o(V,W)$ then $AT \in \mathcal{K}_p^o(Z,W)$ and $\kappa_p^o(AT) \le \pi_p^o(A)\omega_p^o(T)$.
\end{proposition}

\begin{proof}
Let $A : V \to W$ be a completely $p$-summing map.
Let $\mathbf{K}$ be a relatively operator weakly $p$-compact matrix set over $V$, and take $v \in \mathcal{S}_p^w[V]$ such that $\mathbf{K} \subseteq \Theta^v( \mathbf{B}_{\mathcal{S}_{p'}} )$.
By Proposition \ref{prop-p-summing-with-Sp-weak}, $w = (Id_{\mathcal{S}_p}\otimes A)v \in \mathcal{S}_p[W]$ and moreover $\n{w}_{\mathcal{S}_p[W]} \le \pi_p^o(A)\n{v}_{\mathcal{S}_p^w[V]}$.
Observing that $A(\mathbf{K}) \subseteq A\Theta^v(\mathbf{B}_{\mathcal{S}_{p'}}) = \Theta^w( \mathbf{B}_{\mathcal{S}_{p'}} )$ yields that $A(\mathbf{K})$ is operator $p$-compact, and taking the infimum over the $v$'s gives that $\frak m_p^o(A(\mathbf{K})) \le \pi_p^o(A) \frak m_p^{w,o}(\mathbf{K})$.
The result for compositions now follows immediately by taking $\mathbf{K} = T(\mathbf{B}_Z)$.
\end{proof}

In the Banach space setting, there exists a relationship that characterizes $p$-summing maps: a linear map between Banach spaces is $p$-summing if and only if its adjoint maps compact sets to $p$-compact sets \cite[Thm. 3.12]{Delgado-Pineiro-Serrano-adjoints}. For operator spaces, we can establish one implication in general, but the other implication requires an additional assumption. This notion is, in a sense, dual to the well-known concept of exactness: while exactness involves that all finite-dimensional subspaces are uniformly completely isomorphic to subspaces of $M_n$ spaces, the ``structure" we require corresponds to being quotients of $\mathcal{S}_1^n$ spaces. 

\begin{definition}[\cite{webster1998matrix}]
Let $V$ be an operator space.
\begin{enumerate}[(a)]
\item We say that $V$ is \emph{$\lambda$-coexact} if for every $\varepsilon>0$ we can find $n\in\N$ and $W \subseteq \mathcal{S}_1^n$ closed subspace such that $V$ and $\mathcal{S}_1^n/W$ are $(\lambda+\varepsilon)$-completely isomorphic.
\item We say that $V$ is \emph{$\lambda$-subcoexact} if for every finite-dimensional subspace $W \subseteq V$ there exists another finite-dimensional $W_0 \subseteq V$ which contains $W$ and is $\lambda$-coexact. 
\end{enumerate}    
\end{definition}

Before proceeding, we require a couple of preparatory lemmas.
They are both related to the aforementioned fact that any completely $p$-nuclear mapping is completely $p$-summing, and the following norm relation always holds: $\pi_p^o \le \nu_p^o$.
The first one shows that a similar relationship holds between the $\pi_p^o$ and $q\nu_p^o$ norms, whereas the second one (whose classical version can be found, e.g. in \cite[Cor. 9.5]{Tomczak-Jaegermann}) deals with a special case where we have equality between $\pi_p^o$, $\nu_p^o$, and $q\nu_p^o$.

\begin{lemma}\label{lemma:cuasi-p-nuclear es completely p-summing}
Let $V,W$ be operator spaces. If $T \in \mathcal{QN}^o_p(V,W)$ then $T\in \Pi_p^{o}(V,W)$ with $\pi_p^o(T) \le q\nu_p^o(T)$.
\end{lemma}

\begin{proof}
By definition, if $Z$ is an injective operator space and $j: W \to Z$ is a complete isometry, then it follows that $jT: T \to Z$ is completely $p$-nuclear. Consequently, we have $\pi_p^o(jT) \le \nu_p^o(jT) = q\nu_p^o(T)$.
Since the ideal of completely $p$-summing maps is completely injective (see \cite[Comment after Def. 7.5]{ChaDiGa-tensor-norms}), we can conclude that $\pi_p^o(T) \le q\nu_p^o(T)$.
\end{proof}

\begin{lemma}\label{lemma-p-summing-coincides-with-p-nuclear}
If $V$ is an operator space, $1\le p \le \infty$, $n\in\N$ and $T : V \to M_n$ is a linear map, then $\pi_p^o(T) = \nu_p^o(T)$.  
\end{lemma}

\begin{proof}
By Lemma \ref{lemma:cuasi-p-nuclear es completely p-summing}, $\pi_p^o(T) \le q\nu_p^o(T)$.
On the other hand, \cite[Prop. 3.1.3.12]{Junge-Habilitationschrift} gives $ q\nu_p^o(T) \le \pi_p^o(T)$ (though there the notation $\tilde\pi_p(T)$ is used instead of $q\nu_p^o(T)$, see \cite[Def. 3.1.3.5]{Junge-Habilitationschrift}).
Since $T$ has codomain $M_n$ we have $q\nu_p^o(T) = \nu_p^o(T)$, which yields the desired equality. 
\end{proof}

Now we are ready to proof the announced result.

\begin{theorem}\label{thm-adjoint-of-p-summing}
Let $V,W$ be operator spaces, $1\le p < \infty$, and $T : V \to W$ a linear map.
Consider the statements:
\begin{enumerate}[(i)]
    \item $T$ is completely $p$-summing.
    \item $T'$ maps relatively operator compact matrix sets over $W'$ to relatively operator $p$-compact matrix sets over $V'$. 
\end{enumerate}
Then $(i) \Rightarrow (ii)$ always holds, whereas $(ii) \Rightarrow (i)$ does under the additional assumption that $W'$ is $\lambda$-subcoexact for some $\lambda$.
\end{theorem}

\begin{proof}
$(i) \Rightarrow (ii)$: Suppose that $T \in \Pi_p^o(V,W)$. It is enough to prove that $T'(\Theta^{w'}(\mathbf{B}_{\mathcal{S}_1}))$ is relatively operator $p$-compact for any 
$w' = (w'_{ij})_{i,j=1}^\infty \in \mathcal{S}_\infty[W']$.  For a given $w'$, define $S : W \to \mathcal{S}_\infty$ by $w \mapsto (w'_{ij}(w))_{i,j=1}^\infty$.
Note that $S$ is completely bounded with $\|S\|_{\cb}=\|w'\|_{\mathcal{S}_\infty[W']}$ because $\mathcal{S}_\infty[W'] = \mathcal{S}_\infty \otimes_{\min} W' \subseteq \CB(W, \mathcal{S}_\infty)$. We begin by showing that $ST$ is completely $p$-nuclear. For that, consider, for each $N\in\N$, the truncation map $S_N : W' \to M_N$ by $w \mapsto (w'_{ij}(w))_{i,j=1}^N$.
It follows from Lemma \ref{lemma-p-summing-coincides-with-p-nuclear} that $\nu^o_p(S_NT) = \pi^o_p(S_NT) \le \n{S_N}_{\cb}\pi_p^o(T) \le \n{w'}_{\mathcal{S}_\infty[W']} \pi_p^o(T)$.
Therefore, if for every $N\in\N$ we denote by $P_N :\mathcal{S}_\infty \to \mathcal{S}_\infty$ and $j_N : M_N \to \mathcal{S}_\infty$ the truncation and the inclusion onto into the initial $M_N$, respectively, by the same argument as above if $N_2>N_1$, denoting by $j_{N_1,N_2} : M_{N_1} \to M_{N_2}$ the canonical inclusion,
\begin{multline*}
\nu_p^o(P_{N_2}ST - P_{N_1}ST) = 
\nu_p^o\big( j_{N_2}(S_{N_2} -  j_{N_1,N_2}S_{N_1} )T \big) \\
\le \nu_p^o\big( (S_{N_2} - j_{N_1,N_2}S_{N_1} )T \big) \le \n{ S_{N_2}-j_{N_1,N_2}S_{N_1} }_{\cb} \pi_p^o(T).
\end{multline*}
 This shows that the sequence $( P_NST )_{N=1}^\infty$ is Cauchy in $\mathcal{N}_p^o(V,\mathcal{S}_\infty)$, because a matrix in $\mathcal{S}_\infty[W']$ is the limit of its truncations,
so $( P_NST )_{N=1}^\infty$ converges in $\mathcal{N}_p^o(V,\mathcal{S}_\infty)$ to a limit (since it is a complete space).
This limit must be $ST$, which is therefore completely $p$-nuclear, and moreover since $\nu^o_p(P_NST) \le \n{w'}_{\mathcal{S}_\infty[W']} \pi_p^o(T)$ we also have $\nu^o_p(ST) \le \n{w'}_{\mathcal{S}_\infty[W']} \pi_p^o(T)$.

By Proposition \ref{prop:adjoint-of-p-compact} (b), the adjoint $(ST)' : \mathcal{S}_1 \to V'$ is an operator $p$-compact mapping with $\kappa_p^o\big((ST)'\big) \le \n{w'}_{\mathcal{S}_\infty[W']} \pi_p^o(T)$.
Note that for every matrix unit $E_{ij} \in \mathcal{S}_1$ 
we have $(ST)'(E_{ij}) = T'w'_{ij}$, that is, $(ST)'$ is precisely the operator $\Theta^{T'w'}$ from $\mathcal{S}_1$ to $V'$ associated to $T'w'$ 
. Therefore, $(ST)'(\mathbf{B}_{\mathcal{S}_1})$ is an operator $p$-compact matrix set over $V'$ and moreover $\kappa_p^o\left( (ST)'(\mathbf{B}_{\mathcal{S}_1}) \right) \le \n{w'}_{\mathcal{S}_\infty[W']} \pi_p^o(T)$.
Since
\[
T'(\Theta^{w'}(\mathbf{B}_{\mathcal{S}_1})) = \Theta^{T'w'}(\mathbf{B}_{\mathcal{S}_1}) = (ST)'(\mathbf{B}_{\mathcal{S}_1}),
\]
it follows that $T'$ maps relatively operator compact matrix sets over $W'$ to relatively operator $p$-compact matrix sets over $V'$.

$(ii) \Rightarrow (i)$:
Since the ideal $\Pi_p^o$ is maximal \cite[Ex. 8.3.(iii)]{ChaDiGa-tensor-norms}, it suffices to find a uniform estimate on the completely $p$-summing norms of
\[
Q^W_L  T  i^V_M : M \to V \to W \to W/L
\]
for $M \subseteq V$ finite-dimensional and $L \subseteq W$ finite-codimensional,
where $Q^W_L : W \to W/L$ and $i^V_M : M \to V$ are the canonical quotient and inclusion maps, respectively.
Taking adjoints,
\[
(i^V_M)'  T'  (Q^W_L)' : L^\perp \to W' \to V' \to M'.
\]
Note that by assumption, the map
\[
\mathcal{K}_\infty^o(\mathcal{S}_1, W') \to  \mathcal{K}_p^o(\mathcal{S}_1, V'), \qquad S \mapsto T'S 
\]
is well defined.
By a straightforward closed graph argument it is continuous,
so there exists a constant $C>0$ such that for any $S \in \mathcal{K}_\infty^o(\mathcal{S}_1, W')$ we have $\kappa_p^o(T'S)\le C \kappa_\infty^o(S)$.
 
Since $W'$ is $\lambda$-subcoexact for some $\lambda$ and $L^\perp \subset W'$ is finite-dimensional, given $\varepsilon>0$ there exist $N\in\N$ and $\widetilde{L} \subseteq W'$ finite-dimensional containing $L^\perp$ and such that $\widetilde{L}$ is $(\lambda+\varepsilon)$-completely isomorphic to a quotient of $\mathcal{S}_1^N$. So, there exists a complete contraction $\rho : \mathcal{S}_1^N \to \widetilde{L}$ such that 
$\rho(\mathbf{B}_{\mathcal{S}^N_1}) \supseteq (\lambda+\varepsilon)^{-1} \mathbf{B}_{\widetilde{L}}$.
Crucially, note that $\rho$ is operator compact.
Indeed, since $\mathcal{S}_1^N$ is finite-dimensional and $\CB(\mathcal{S}_1^N,W) \equiv M_N(W)$ 
then $\rho = \Theta^w$ for some $w \in M_N(W)$ of norm one, from where it follows that $\kappa_\infty^o(\rho) = \n{\rho}_{\cb} = 1$.
Now define $Q = i^{W'}_{\widetilde{L}}\rho P_N : \mathcal{S}_1 \to W'$ where $P_N : \mathcal{S}_1 \to \mathcal{S}_1^N$ is the projection onto the initial $N\times N$ block.
Since $Q$ is a finite-rank map it is operator compact 
, and therefore by the choice of $C$ we have $\kappa_p^o(T'(Q^W_L)'Q) \le C \kappa_\infty^o((Q^W_L)'Q) \le  C \kappa_\infty^o(Q) \le C  \kappa_\infty^o(\rho) =C $.
Therefore, there exists $x'\in \mathcal{S}_p[V']$ such that $\n{x'}_{\mathcal{S}_p[V']} \le (C+\varepsilon)$ and $(T'(Q^W_L)'Q)(\mathbf{B}_{\mathcal{S}_1}) \subseteq \Theta^{x'}(\mathbf{B}_{\mathcal{S}_p'})$.
Since $Q(\mathbf{B}_{\mathcal{S}_1}) \supseteq (\lambda+\varepsilon)^{-1} \mathbf{B}_{L^\perp}$,
it then follows that $T'(Q^W_L)'(\mathbf{B}_{L^\perp}) \subseteq (\lambda+\varepsilon)\Theta^{x'}(\mathbf{B}_{\mathcal{S}_p'})$.
This implies
$\kappa_p^o( T'(Q^W_L)') \le (\lambda+\varepsilon)\n{x'}_{\mathcal{S}_p[V']}$ and therefore
$\kappa_p^o( (i^V_M)'T'(Q^W_L)') \le (\lambda+\varepsilon)\n{x'}_{\mathcal{S}_p[V']}$.
It then follows from Proposition \ref{prop:adjoint-of-p-compact}.(b)
that the mapping $Q^W_L T  i^V_M \in \mathcal{QN}^o_p(M,W/L)$ and also $q\nu_p^o( Q^W_L T  i^V_M ) \le (\lambda+\varepsilon)\n{x'}_{\mathcal{S}_p[V']}$.
Additionally, according to Lemma \ref{lemma:cuasi-p-nuclear es completely p-summing}, $Q^W_L T i^V_M$ is completely $p$-summing, with $\pi_p^o(Q^W_L T i^V_M) \le (\lambda+\varepsilon)\n{x'}{\mathcal{S}_p[V']} \le (\lambda+\varepsilon)(C+\varepsilon)$, thus concluding the proof.

\end{proof}

\begin{remark}
The proof above shows that if $T : V \to W$ is completely $p$-summing and $\mathbf{K}$ is an operator compact matrix set over $W'$, then $\frak m_p^o(T'(\mathbf{K})) \le \frak m_\infty^o(\mathbf{K}) \pi_p^o(T)$.
Also, the map $\mathcal{S}_\infty[W'] \to \mathcal{K}_p^o(\mathcal{S}_1,V')$ given by $A \mapsto (AT)'$ is well-defined and has norm at most $\pi^o_p(T)$.
\end{remark}

\section{Regularity of the ideal of operator $p$-compact mappings}\label{sec-regularity}

Recall that a mapping ideal is called \emph{regular} if $T : V \to W$ is in the ideal if and only if $\iota_W \circ T$ also is, where $\iota_W : W \to W''$ is the canonical injection.
In the Banach case, the ideal of $p$-compact mappings is regular \cite[Cor. 3.6]{Delgado-Pineiro-Serrano-adjoints}, even with equality of the norms \cite[Cor. 2.6]{Galicer-Lassalle-Turco}. As a consequence, a subset of a Banach space is relatively $p$-compact if and only if it is relatively $p$-compact in the bidual \cite[Thm. 2.4]{Galicer-Lassalle-Turco}.

In the operator space setting, we are not able to obtain the same type of results in full generality. However, in this section we prove analogues of those results but with additional conditions such as local reflexivity, which is of course not a surprise.
Our approach follows that of \cite{Pietsch-p-compact}.

We begin with several lemmas. The first of these is proved similarly as in  \cite[Lem. E.3.2]{Pietsch-Operator-Ideals}.

\begin{lemma}\label{lemma-Pietsch-E-3-2}
Suppose that $V_0$ is finite-dimensional, $A \in \CB(V', V_0')$, $T \in \mathcal{F}(W',V')$, and $V$ is strongly locally reflexive. Then, given $\varepsilon>0$ there exists $S \in \CB(V_0,V)$ such that $\n{S}_{\cb} \le (1+\varepsilon)\n{A}_{\cb}$ and $S'T = AT$.   
\end{lemma}

\begin{lemma}\label{lemma-Pietsch-10-2-5}
Suppose that $W$ has CMAP, $T \in \mathcal{F}(V,W)$  and $\varepsilon>0$. Then there exists $A \in\mathcal{F}(W,W)$ such that $\n{A}_{\cb} \le 1+\varepsilon$ and $AT=T$.  
\end{lemma}

\begin{proof}
For the finite-dimensional subspace $G=T(V)$ of $W$ we have that the identity on $G$ can be written as $I_G=\sum_{j=1}^n w_j'\otimes w_j$, for certain $w_1,\dots w_n\in G$, $w_1'\dots, w_n'\in W'$. Since $W$ has CMAP, there  exists a mapping $S \in \mathcal{F}(W,W)$ whose $\cb$-norm is less than or equal to 1 and such that $\sum_{j=1}^n \|w_j'\| \|w_j - Sw_j\|\le \varepsilon$.
By standard perturbation arguments \cite[Lem. 2.13.2]{Pisier-OS-theory}, there exists a mapping $A \in \mathcal{F}(W,W)$ whose $\cb$-norm is less than $1+\varepsilon$ and which is the identity on $G$, and we are done.     
\end{proof}

\begin{lemma}\label{lemma-Pietsch-10-2-6}
Suppose that $V'$ has CMAP and $V$ is strongly locally reflexive. Let $T \in \mathcal{F}(V,W)$ and $\varepsilon>0$.
Then there exists $S \in \mathcal{F}(V,V)$ such that $\n{S}_{\cb} \le 1+\varepsilon$ and $TS=T$.
\end{lemma}

\begin{proof}
Let $\delta>0$ such that $(1+\delta)^2 \le 1+\varepsilon$.
By Lemma \ref{lemma-Pietsch-10-2-5} applied to $T'$, we can find $A \in \mathcal{F}(V',V')$ such that $\n{A}_{\cb} \le 1+\delta$ and $AT' =T'$.
By Lemma \ref{lemma-Pietsch-E-3-2}, there exists $S \in\mathcal{F}(V,V)$ such that $\n{S}_{\cb} \le (1+\delta)\n{A}_{\cb}$ and $S'T' = T'$.
It follows that $\n{S}_{\cb} \le 1+\varepsilon$ and $TS=T$.
\end{proof}

Every finite rank mapping $T : V \to W$ is clearly completely right $p$-nuclear. In this case, we have a finite version of the right $p$-nuclear norm which considers just factorizations through finite dimensional Schatten spaces. So, we define $\nu^{p,\fin}_o(T)$ as the infimum of $\n{R}_{\cb} \n{S}_{\cb} \n{a}_{\mathcal{S}^n_{2p}} \n{b}_{\mathcal{S}^n_{2p}}$ over all factorizations of the form
\[
\xymatrix{
V \ar[r]^T \ar[d]_{R} &W \\
\mathcal{S}^n_{p'} \ar[r]_{M(a,b)} &\mathcal{S}^n_1 \ar[u]_{S}
}
\]
where $a,b\in M_n$ and $n\in\N$. If the domain is finite-dimensional, both norms coincide:

\begin{lemma}\label{lemma-finite-right-p-nuclear}
If $V_0$  is finite-dimensional and $T \in \CB(V_0,W)$, then $\nu^{p,\fin}_o(T) = \nu^p_o(T)$.   
\end{lemma}

\begin{proof}
It is clear that $\nu^p_o(T) \le \nu^{p,\fin}_o(T)$, since calculating $\nu^p_o(T)$ allows for more factorizations. On the other hand, starting with a completely right $p$-nuclear factorization for $T$
\[
\xymatrix{
V_0 \ar[r]^T \ar[d]_{R} &W \\
\mathcal{S}_{p'} \ar[r]_{M(a,b)} &\mathcal{S}_1 \ar[u]_{S},
}
\]
note that $R(V_0)$ is a finite-dimensional subspace of $\mathcal{S}_{p'}$. Therefore, given $\varepsilon>0$ for $n \in \N$ large enough we have that the projection from $\mathcal{S}_{p'}$ onto the initial $\mathcal{S}^n_{p'}$
 is an $(1+\varepsilon)$ complete isomorphism when restricted to $R(V_0)$.
 This implies $\nu^{p,\fin}_o(T) \le \nu^p_o(T)$.
\end{proof}

For the  ideal of completely nuclear mappings, it is shown in  \cite[Lem. 12.2.7]{Effros-Ruan-book} that the unit ball is point-weak closed, when the domain is finite-dimensional.
With a similar argument the same can be deduced for any mapping ideal.   We include its proof for completeness.

\begin{lemma}\label{lemma-ball-is-point-weak-closed}
Let $(\mathfrak{A},\|\cdot\|_{\mathfrak{A}})$ be a mapping ideal. If $V_0$  is finite-dimensional, then for any operator space $W$ the unit ball of $\mathfrak{A}(V_0,W)$ is point-weak closed.    
\end{lemma}

\begin{proof}
We first claim that the point-weak closure of  $B_{\mathfrak{A}(V_0,W)}$ concides with its point-norm closure.
Suppose that $T : V_0 \to W$ is the point-weak limit of a net $T_\alpha \in B_{\mathfrak{A}(V_0,W)}$.
For fixed $v_1, \dotsc, v_n \in V_0$, define $\Phi : \mathfrak{A}(V_0,W) \to W \oplus_\infty \cdots \oplus_\infty W$ by
\[
\Phi(S) = ( S(v_1), \dotsc, S(v_n) ).
\]
Since $T_\alpha \to T$ in the point-weak topology, it is clear that  $( T_\alpha(v_1), \dotsc, T_\alpha(v_n) ) \to ( T(v_1), \dotsc, T(v_n) )$ in the weak topology of $W \oplus_\infty \cdots \oplus_\infty W$.
Therefore, $( T(v_1), \dotsc, T(v_n) )$ is in the weak closure of $\Phi(B_{\mathfrak{A}(V_0,W)})$. Since the latter is a convex set, it follows from the classical Mazur's Theorem that $( T(v_1), \dotsc, T(v_n) )$ is in the norm closure of $\Phi(B_{\mathfrak{A}(V_0,W)})$. That is, for any $\varepsilon>0$ there exists $S \in B_{\mathfrak{A}(V_0,W)}$ such that $\n{S(v_j)-T(v_j)}<\varepsilon$ for each $j=1,\dotsc,n$.
We conclude that $T$ is in the point-norm closure of $B_{\mathfrak{A}(V_0,W)}$.

Suppose then that $T : V_0 \to W$ is a point-norm limit of a net $S_\alpha\in B_{\mathfrak{A}(V_0,W)}$.
Fix a basis $\{ v_1, \cdots, v_d\}$ of $V_0$ with biorthogonal basis $\{v'_1, \dotsc, v'_d\} \subset V_0'$.
Under the identification $\CB(V_0,W) \equiv V_0' \otimes W$, we can write
\[
S_\alpha = \sum_{j=1}^d v_j' \otimes w_{\alpha,j}, \qquad T = \sum_{j=1}^d v_j' \otimes w_{j}
\]
where $w_{\alpha,j} = S_\alpha(v_j)$ and $w_j = T(v_j)$.
Since $(S_\alpha)$ converges to $T$ in the point-norm topology, for each $1 \le j \le d$ we have
\[
\n{w_{\alpha,j} - w_j} = \n{S_\alpha(v_j) - T(v_j)} \to 0.
\]
Therefore,
\[
\|S_\alpha-T\|_{\mathfrak{A}} = \Big\| \sum_{j=1}^d v_j' \otimes (w_{\alpha,j} - w_j)  \Big\|_{\mathfrak{A}} \le \sum_{j=1}^d \| v_j' \otimes (w_{\alpha,j} - w_j) \|_{\mathfrak{A}} =   \sum_{j=1}^d \n{ v_j'}\n{ w_{\alpha,j} - w_j} \to 0,
\]
and from
\[
\|T\|_{\mathfrak{A}} \le \|S_\alpha\|_{\mathfrak{A}} + \|S_\alpha-T\|_{\mathfrak{A}} \le 1 + \|S_\alpha-T\|_{\mathfrak{A}} 
\]
we then conclude $\|T\|_{\mathfrak{A}} \le 1$ as desired.
\end{proof}

From the previous two lemmas we deduce a first step towards regularity for completely right $p$-nuclear mappings.

\begin{lemma}\label{lemma-regularity-right-p-nuclear-fin-dim-domain}
If $V_0$  is finite-dimensional, $W$ is locally reflexive, and $T \in \CB(V_0,W)$, then $\nu^p_o(T) = \nu^p_o(\iota_WT)$.   
\end{lemma}

\begin{proof}
Let $\varepsilon>0$. From Lemma \ref{lemma-finite-right-p-nuclear}, find a factorization
\[
\xymatrix{
V_0 \ar[r]^T \ar[d]_{R} &W \ar[r]^{\iota_W} &W''\\
\mathcal{S}^n_{p'} \ar[r]_{M(a,b)} &\mathcal{S}^n_1 \ar[ur]_{S} 
}
\]
with $\n{R}_{\cb} \n{S}_{\cb} \n{a}_{\mathcal{S}^n_{2p}} \n{b}_{\mathcal{S}^n_{2p}} \le (1+\varepsilon) \nu^p_o(\iota_W T)$.
Since $W$ is locally reflexive, there is a net of maps $S_i \in \CB(\mathcal{S}_1^n,W)$ with $\n{S_i}_{\cb} \le \n{S}_{\cb}$ such that $\iota_WS_i$ converges to $S$ in the point-weak$^*$ topology.
Therefore $\iota_WS_iM(a,b)R$ converges to $\iota_WT$ in the  point-weak$^*$ topology, which means that  $S_iM(a,b)R$ converges to $T$ in the point-weak topology.
Since 
\[
\nu^p_o(S_iM(a,b)R) \le \n{R}_{\cb} \n{S_i}_{\cb} \n{a}_{\mathcal{S}^n_{2p}} \n{b}_{\mathcal{S}^n_{2p}} \le (1+\varepsilon) \nu^p_o(\iota_W T),
\]
it follows from Lemma \ref{lemma-ball-is-point-weak-closed} that $\nu^p_o(T) \le \nu^p_o(\iota_W T)$.
Since the opposite inequality holds by the ideal property, the desired conclusion follows.
\end{proof}

With hypothesis we can see that composing a complety right $p$-nuclear mapping with the canonical inclusion on the bidual preserves the ideal norm.

\begin{proposition}\label{prop-regularity-right-p-nuclear-weakeer-assumption}
Suppose that $V'$ has CMAP, $V$ is strongly locally reflexive, and $W$ is locally reflexive.
Then the mapping $T \mapsto \iota_W T$ is an isometry from 
$\mathcal{N}^p_o(V,W)$ into $\mathcal{N}^p_o(V,W'')$.
\end{proposition}

The proof of the previous result follows from the following two lemmas and the fact that finite rank mappings are dense in $\mathcal{N}^p_o(V,W)$.

\begin{lemma}\label{lemma-Pietsch-6-8-4-modified}
Let $A \in\mathcal{F}(V_0,V)$ and $T \in \mathcal{N}^p_o(V,W)$.
If $W$ is locally reflexive then
$\nu^p_o(TA) \le \nu^p_o(\iota_WT) \n{A}_{\cb}$.    
\end{lemma}

\begin{proof}
Consider the factorization $A = JA_0$ where $A_0 : V_0 \to A(V_0)$ is just $A$ with a smaller codomain, and $J : A(V_0) \to V$ is the inclusion. Then, since $\n{A}_{\cb} = \n{A_0}_{\cb}$ and using Lemma \ref{lemma-regularity-right-p-nuclear-fin-dim-domain},
\[
\nu^{p}_o(TA) = \nu^{p}_o(TJA_0) \le \nu^{p}_o(TJ)\n{A_0}_{\cb} = \nu^{p}_o(\iota_W TJ) \n{A}_{\cb} \le  \nu^p_o(\iota_W T) \n{A}_{\cb}.
\]
\end{proof}

\begin{lemma}\label{lemma-Pietsch-10-3-1-modified}
Suppose that $V'$ has CMAP, $V$ is strongly locally reflexive, and $W$ is locally reflexive.
Then for any $T \in \mathcal{F}(V,W)$ we have $\nu^{p}_o(T) = \nu^p_o(\iota_WT)$.
\end{lemma}

\begin{proof}
Given $\varepsilon>0$, by Lemma \ref{lemma-Pietsch-10-2-6} there exists $A \in \mathcal{F}(V,V)$   such that $\n{A}_{\cb} \le 1+\varepsilon$ and $TA =T$.
Now, by Lemma \ref{lemma-Pietsch-6-8-4-modified},
\[
\nu^{p}_o(T) = \nu^{p}_o(TA) \le \nu^p_o(\iota_W T) \n{A}_{\cb} \le (1+\varepsilon) \nu^p_o(\iota_W T),
\]
and therefore $\nu^{p}_o(T) \le \nu^p_o(\iota_W T)$.
Since the opposite inequality always holds, we get the desired conclusion.
\end{proof}

An operator space $V$ is said to be \emph{$N$-maximal} (resp. \emph{$N$-minimal}) if for every operator space $W$ and every linear map $T : V \to W$ (resp. every linear map $T : W \to V$) we have $\n{T : V \to W}_{\cb} = \n{T_N : M_N(V) \to M_N(W)}$ (resp. $\n{T : W \to V}_{\cb} = \n{T_N : M_N(W) \to M_N(V)}$),
see \cite{OikhbergRicard2004MathAnn, LehnerPhDThesis}.
Spaces of the form $\ell_1(I; \mathcal{S}_1^N)$ are $N$-maximal: one can either argue by duality \cite[Lem. 2.4]{OikhbergRicard2004MathAnn} using the fact that $\ell_\infty(I; M_N)$ is clearly $N$-minimal by Smith's lemma \cite[Prop. 1.12]{Pisier-OS-theory}, or prove it directly.

\begin{lemma}\label{lemma-N-maximal-is-quotient-of-ell_1({S}_1^N)}
The operator space $V$ is $N$-maximal if and only if $V$ is a complete quotient of a space of the form $\ell_1(I;\mathcal{S}_1^N)$.  
\end{lemma}

\begin{proof}
Suppose that $V$ is $N$-maximal. Following the same construction as in \cite[Prop. 2.1.2.2]{Pisier-OS-theory} but stopping at level $N$, we get a space of the form $\ell_1(I;\mathcal{S}_1^N)$ and a complete contraction $Q : \ell_1(I;\mathcal{S}_1^N) \to V$ such that its $N$-th amplification $Q_N$ is a 1-quotient.
Recalling that $\ell_1(I;\mathcal{S}_1^N)$ is $N$-maximal, note that for any linear map $T : V \to W$ we have
\begin{multline*}
\n{T : V \to W}_{\cb} = \n{T_N : M_N(V) \to M_N(W)} = \n{T_N Q_N : M_N(\ell_1(I;\mathcal{S}_1^N)) \to M_N(W)}\\
= \n{(TQ)_N : M_N(\ell_1(I;\mathcal{S}_1^N)) \to M_N(W)}
= \n{TQ : \ell_1(I;\mathcal{S}_1^N) \to W}_{\cb},
\end{multline*}
which implies that $Q$ is a complete quotient.

The converse is straightforward: complete quotients of $N$-maximal spaces are themselves $N$-maximal.
\end{proof}

From all of the above, we conclude:

\begin{proposition}
If the operator space $V$ is $N$-maximal and $W$ is locally reflexive, then the mapping $T \mapsto \iota_W T$ is an isometry from 
$\mathcal{K}_p^o(V,W)$ into $\mathcal{K}_p^o(V,W'')$.
\end{proposition}

\begin{proof}
 Since $V$ is $N$-maximal, there is a complete contraction $Q : \ell_1(I;\mathcal{S}_1^N) \to V$. Given that $\ell_1(I;\mathcal{S}_1^N)$ is a projective operator space, for  $T \in \mathcal{K}_p^o(V,W)$  we have the equality $\kappa_p^o(T)=\nu^p_p(TQ)$ \cite[Prop. 3.8.]{ChaDiGa-Operator-p-compact}. Since $\ell_1(I;\mathcal{S}_1^N)$ is strongly locally reflexive because its dual is the von Neumann algebra $\ell_\infty(I;M_N)$ \cite[Thm. 15.3.5]{Effros-Ruan-book}, and moreover $\ell_\infty(I;M_N)$ has the CMAP by Lemma \ref{lemma-ell_infty-has-CMAP}, we can use Proposition \ref{prop-regularity-right-p-nuclear-weakeer-assumption} to get
 $$\kappa_p^o(T)=\nu^p_p(TQ)=\nu^p_p(\iota_W TQ)=\kappa_p^o(\iota_WT),$$
 which gives the desired equality.
\end{proof}

Recall that given operator spaces $V$ and $W$, a linear map $T : V \to W$ is said to be \emph{completely integral} if
\[
\iota^o(T) = \sup\big\{ \nu^o(\restr{T}{V_0}) \;:\; V_0 \subseteq V \text{ finite-dimensional} \big\}
\]
is finite. The set of all such maps is denoted by $\mathcal{I}^o(V,W)$, and this is a mapping ideal \cite[Sec. 12.3]{Effros-Ruan-book}.

Under certain condition the completely nuclear norm and completely integral norms coincide. 

\begin{lemma}\label{lemma-nuclear-coincides-integral}
Suppose that $W$ has CMAP and $V$ is  locally reflexive.
Then for any $T \in \mathcal{F}(V,W)$ we have $\nu^o(T) = \iota^o(T)$.
\end{lemma}

\begin{proof}
Given $\varepsilon>0$, by Lemma \ref{lemma-Pietsch-10-2-5} there exists $A\in\mathcal F(W,W)$ such that $\n{A}_{\cb} \le 1+\varepsilon$ and $AT=T$.
Let $A_0:W \to A(W)$ be the map given by $A_0w=Aw$, and let $J:A(W) \to W$ the formal inclusion so that $T=JA_0T$.
Now,
$$\nu^o(T)=\nu^o(JA_0T)\leq\nu^o(A_0T) = \iota^o(A_0T), $$
where the last equality follows from the fact that $V$ is locally reflexive \cite[Thm. 14.3.1]{Effros-Ruan-book}. The ideal property now shows that $\iota^o(A_0T) \leq \n{A_0}_{\cb} \iota^o(T) \le (1+\varepsilon)\iota^o(T)$,
which leads to $\nu^o(T) \leq \iota^o(T)$. Since the other inequality always holds, we have the result.
\end{proof}

We now state the main theorem of this section which corresponds to a more involved version of \cite[Lem. 3]{Pietsch-p-compact}.

\begin{theorem}\label{theorem-Pietsch}
Suppose that $V'$ has CMAP, $V$ is strongly locally reflexive, and both $V'$ and $W$ are locally reflexive.
If for $T\in\CB(V,W)$ we have that $\iota_WT \in \mathcal{N}^p_o(V,W'')$, then $T\in \mathcal{N}^p_o(V,W)$ and moreover $\nu^p_o(T) = \nu^p_o(\iota_WT)$.
\end{theorem}

\begin{proof}
We already know from Proposition \ref{prop-regularity-right-p-nuclear-weakeer-assumption} that the mapping $T \mapsto \iota_W T$ is an isometry from 
$\mathcal{N}^p_o(V,W)$ into $\mathcal{N}^p_o(V,W'')$. 
Suppose that there exists $T_0 \in \CB(V,W)$ such that $\iota_WT_0 \in \mathcal{N}^p_o(V,W'')$ but $T_0\not\in \mathcal{N}^p_o(V,W)$.
Without loss of generality, we may assume $\nu^p_o(\iota_WT_0) = 1$.
By the Hahn-Banach theorem, there exists a continuous functional $\varphi_0 : \mathcal{N}^p_o(V,W'') \to \C$ such that $\varphi_0( \iota_WT_0  ) = 1$, but $\varphi_0( \iota_WT  )=0$ for any $T \in \mathcal{N}^p_o(V,W)$.
We now define an associated mapping $T_{\varphi_0} : W'' \to V''$ via
\[
\pair{T_{\varphi_0} w''}{v'} =  \varphi_0( v' \otimes w'' ), \qquad v' \in V', w'' \in W''.
\]

Let us now show that this mapping is completely bounded. Indeed, for $n\in \mathbb{N}$ consider
$(T_{\varphi_0})_n: M_n(W'') \to M_n(V'')$ its $n$-amplification.
Thus, for any $(w_{ij}'') \in M_n(W'') = M_n \otimes W''$, using the matrix pairing $\langle\langle\ , \rangle\rangle$  defined in  \cite[(1.1.27)]{Effros-Ruan-book} we have
\begin{align*}
\left\| (T_{\varphi_0})_n (w_{ij}'')\right\| & = \sup_{\|(v_{kl}')\|_{M_n(V')}\le 1}  \left\|  \langle\langle (T_{\varphi_0})_n (w_{ij}''), (v_{kl}') \rangle\rangle\right\|\\
& = \sup_{\|(v_{kl}')\|_{M_n(V')}\le 1}  \left\| \left(\varphi_o(w_{ij}'' \otimes v_{kl}')\right)_{i,j,k,l}\right\|_{M_{n^2}} \\
& \le \|\varphi_0\| \cdot \left\| \left(w_{ij}'' \otimes v_{kl}'\right)_{i,j,k,l}\right\|_{M_{n^2}(\mathcal{N}^p_o(V,W''))} \\
& = \|\varphi_0\| \cdot \| (w_{ij}'')\|_{M_n(W'')} \|(v_{kl}')\|_{M_n(V')},
\end{align*}
where the last equality follows from the definition of mapping ideal for $\mathcal{N}^p_o$ (see \cite[Def. 7.1. (b) and Rmk. 7.2 (vii)]{ChaDiGa-tensor-norms}).

Now, note that for any $v' \in V', w \in W$ we have 
$
\pair{T_{\varphi_0} \iota_Ww}{v'} = \varphi_0( v' \otimes \iota_Ww ) = 0
$
since $v' \otimes w \in \mathcal{N}^p_o(V,W)$.
Therefore $T_{\varphi_0}\iota_W = 0$, and thus $\iota_W'T_{\varphi_0}' = 0$.
Coming back to our original mapping, due to $\iota_WT_0 \in \mathcal{N}^p_o(V,W'')$ we can consider a completely right $p$-nuclear factorization
\[
\xymatrix{
V \ar[r]^{T_0} \ar[d]_{\alpha} &W \ar[r]^{\iota_W} &W''\\
\mathcal{S}_{p'} \ar[r]_{M(a,b)} &\mathcal{S}_1 \ar[ur]_{\beta}, 
}
\]
where $a,b\in \mathcal{S}_{2p}$. We may assume without loss of generality that $\n{\alpha}_{\cb}, \n{\beta}_{\cb} \le 1$.
Define for each $n\in\N$ the mapping $\gamma_n = \beta M(a_n,b_n)\alpha : V \to W''$ where $a_n \in \mathcal{S}_{2p}$ is the restriction of $a$ to the initial $n\times n$ block, and similarly for $b_n$.
Note that $(a_n)$ converges to $a$ in $\mathcal{S}_{2p}$, and  $(b_n)$ converges to $b$ in $\mathcal{S}_{2p}$.
Using the same argument as in \cite[Prop. 2.7]{ChaDiGa-Operator-p-compact}, we get that 
\[
\n{\gamma_n - \gamma_m}_{\mathcal{N}^p_o(V,W'')} \le \n{a_n -a_m} \n{b} + \n{a}\n{b_n-b_m}.
\]
Therefore $(\gamma_n)$ is Cauchy in $\mathcal{N}^p_o(V,W'')$ and thus has a limit there. But convergence in $\mathcal{N}^p_o(V,W'')$ implies pointwise convergence, so said limit must be $\iota_W T_0$. Thus, we have that $\lim_{n\to\infty}\nu^p_o( \iota_WT_0 - \gamma_n ) = 0$.

Note that $\gamma_n$ is a mapping of finite rank.
For each $i,j\in\N$, let $\alpha'_{ij} \in V'$ be the functional that assigns to each $v\in V$ the $ij$-th entry of $\alpha(v)$. If we denote by $E_{rs}$ the matrix units in $M_n$, then it is straightforward to verify that
\[
\gamma_n = \sum_{i,j,r,s=1}^n a_{ri}b_{js} \alpha'_{ij} \otimes \beta E_{rs},
\]
where of course the numbers $a_{ij}$ and $b_{js}$ are the entries in the matrix representations of $a$ and $b$, respectively.
A calculation then shows
\begin{equation}\label{eqn-Pietsch-trace-gamma_n}
\tr\big( \gamma_n' T_{\varphi_0}' \iota_{V'}\big) = \sum_{i,j,r,s=1}^n a_{ri}b_{js} \varphi_0\big(\alpha'_{ij} \otimes \beta E_{rs}\big) = \varphi_0\bigg( \sum_{i,j,r,s=1}^n a_{ri}b_{js} \alpha'_{ij} \otimes \beta E_{rs}\bigg) = \varphi_0(\gamma_n).
\end{equation}
where the above trace makes sense because the mapping involved belongs to $\mathcal{F}(V',V')$.

We next will show that the sequence $(\gamma_n' T_{\varphi_0}' \iota_{V'})$ is Cauchy in $\mathcal{N}^o(V',V')$. 
For a mapping $A \in \CB(V',V')$ with $\n{A}_{\cb} \le 1$ and $m < n$, 
\begin{multline}\label{eqn-Pietsch-nuclear}
\big|\tr\big( A\gamma_n' T_{\varphi_0}' \iota_{V'} - A \gamma_m' T_{\varphi_0}' \iota_{V'}\big)\big| = \\
\left| \varphi_0\left( \sum_{i,j,r,s=1}^n a_{ri}b_{js} A\alpha'_{ij} \otimes \beta E_{rs} - \sum_{i,j,r,s=1}^m a_{ri}b_{js} A\alpha'_{ij} \otimes \beta E_{rs} \right) \right|  
\\
\le \n{\varphi_0} \nu^p_o\left( \sum_{i,j,r,s=1}^n a_{ri}b_{js} A\alpha'_{ij} \otimes \beta E_{rs} - \sum_{i,j,r,s=1}^m a_{ri}b_{js} A\alpha'_{ij} \otimes \beta E_{rs} \right)
\end{multline}
Recall that the truncation of $\alpha$ (which, for simplicity, we still call the same) is $\alpha : V \to \mathcal{S}_{p'}^n$ given by $v \mapsto (\alpha'_{ij}(v))_{i,j=1}^n$.
From the identification $\CB(V, \mathcal{S}_{p'}^n) =  V' \otimes_{\min}\mathcal{S}_{p'}$, since $A$ is a complete contraction we have that the map $\alpha^A : V \to \mathcal{S}_{p'}^n$ given by  $v \mapsto \big((A\alpha'_{ij})(v)\big)_{i,j=1}^n$ is completely bounded with $\n{\alpha^A}_{\cb} \le \n{\alpha}_{\cb}$.
Thus, if we write
\[
\gamma_n^A = \sum_{i,j,r,s=1}^n a_{ri}b_{js} A\alpha'_{ij} \otimes \beta E_{rs},
\]
the same argument we used above for $(\gamma_n)$ will show that
\begin{equation}\label{eqn-Pietsch-gammaA}
\n{\gamma_n^A - \gamma_m^A}_{\mathcal{N}^p_o(V,W'')} \le  \n{a_n -a_m} \n{b} + \n{a}\n{b_n-b_m}.    
\end{equation}
From equations \eqref{eqn-Pietsch-nuclear} and \eqref{eqn-Pietsch-gammaA}, together with the trace duality between the completely bounded and completely integral norms, we get that  $(\gamma_n' T_{\varphi_0}' \iota_{V'})$ is Cauchy in $\mathcal{I}^o(V',V')$, so by  Lemma \ref{lemma-nuclear-coincides-integral} it is also Cauchy in ${\mathcal{N}^o(V',V')}$
and thus it converges to a limit in $\mathcal{N}^o(V',V')$. Since convergence in $\mathcal{N}^o(V',V')$ implies pointwise convergence, this limit must be $T_0'\iota_W'T_{\varphi_0}' \iota_{V'}$. In particular, since $V'$ has CMAP, it follows that the trace of $T_0'\iota_W'T_{\varphi_0}' \iota_{V'}$ is well defined and is the limit of $(\tr(\gamma_n' T_{\varphi_0}' \iota_{V'}))$.

Now, using the fact that $\gamma_n \to \iota_WT_0$ in $\mathcal{N}^p_o(V,W'')$ and \eqref{eqn-Pietsch-trace-gamma_n},
\[
1 =  \varphi_0( \iota_WT_0  ) = \lim_{n\to\infty} \varphi_0( \gamma_n ) = \lim_{n\to\infty} \tr\big( \gamma_n' T_{\varphi_0}' \iota_{V'}\big) = \tr(T_0'\iota_W'T_{\varphi_0}' \iota_{V'}) = 0,
\]
where the last equality follows from $\iota_W'T_{\varphi_0}' = 0$, so we have obtained a contradiction.
\end{proof}

From the previous theorem we obtain an operator space version of \cite[Cor. 2.6]{Galicer-Lassalle-Turco} (see also \cite[Cor. 3.6]{Delgado-Pineiro-Serrano-adjoints}), but unlike in the Banach space case we do require some assumptions on the operator spaces.

\begin{corollary}\label{cor-compactness-in-the-bidual-for-operators}
If $V$ is an $N$-maximal operator space, $W$ is locally reflexive, and $T\in\CB(V,W)$ satisfies that $\iota_WT\in\mathcal{K}_p^o(V,W'')$, then  $T\in\mathcal{K}_p^o(V,W)$ and moreover $\kappa_p^o(T) = \kappa_p^o(\iota_WT)$.
\end{corollary}

\begin{proof}
Since $V$ is $N$-maximal, by Lemma \ref{lemma-N-maximal-is-quotient-of-ell_1({S}_1^N)} we get a complete quotient $Q : \ell_1(I;\mathcal{S}_1^N) \to V$.
Since $\ell_1(I;\mathcal{S}_1^N)$ is a projective operator space and $\iota_WT\in\mathcal{K}_p^o(V,W'')$, by \cite[Prop. 3.7]{ChaDiGa-Operator-p-compact} we have that $\iota_WTQ \in\mathcal{N}_o^p(\ell_1(I;\mathcal{S}_1^N),W'')$ and $\nu_o^p(\iota_WTQ) = \kappa_p^o(\iota_WT)$.
Recall that $\ell_1(I;\mathcal{S}_1^N)$ is strongly locally reflexive  and  $\ell_\infty(I;M_N)$ has the CMAP.
Moreover, since $\ell_\infty(I;M_N)$ is $N$-minimal it follows from \cite[Prop. 18.4]{Pisier-OS-theory} that it is locally reflexive.
Thus, we can apply Theorem \ref{theorem-Pietsch} to get that  $TQ \in\mathcal{N}_o^p(\ell_1(I;\mathcal{S}_1^N),W)$ and $\nu_o^p(TQ) = \nu_o^p(\iota_WTQ) = \kappa_p^o(\iota_WT)$.
From \cite[Prop. 3.8]{ChaDiGa-Operator-p-compact} we can now conclude that $T\in\mathcal{K}_p^o(V,W)$ and $\kappa_p^o(T)= \kappa_p^o(\iota_WT)$. 
\end{proof}
\section{Associating operator $p$-compact maps to operator $p$-compact matrix sets} \label{sec: associating}

Just from the definitions, each operator $p$-compact map $T: V \to W$ is naturally associated to a relatively operator $p$-compact matrix set: $T(\mathbf{B}_V)$.
It turns out that the opposite is also true, and every relatively operator $p$-compact matrix set is associated with an operator $p$-compact map.
The analogous result in the Banach case can be found in \cite[Prop. 3.5]{Delgado-Pineiro-Serrano-adjoints}.

To achieve this we need first to develop some theory about a noncommutative notion of absolute convexity for matrix sets.

Following \cite{Effros-Webster}, we say that a matrix set $\mathbf{K} = (K_n)$ over $V$ is \emph{absolutely matrix convex} if:
\begin{enumerate}[(i)]
	\item For all $x \in K_n$ and $y \in K_m$, $x \oplus y \in K_{m+n}$.
	\item For all $x \in K_n$, $a \in M_{m,n}$, and $b \in M_{n,m}$ with $\n{a}, \n{b} \le 1$,
	$a x b \in K_m$.
\end{enumerate}

The main example of an absolutely matrix convex set of matrices over $V$ comes from the operator space structure: it follows immediately from Ruan's axioms that both the open and closed matrix unit balls of $V$ are  absolutely matrix convex.

Clearly, the intersection of a family of absolutely matrix convex sets over $V$ (understood as intersecting on each matrix level) is again absolutely matrix convex. Therefore we can define the \emph{absolutely matrix convex hull} of a given matrix set $\mathbf{K}$ over $V$, denoted by $\amconv(\mathbf{K})$,
as the smallest absolutely matrix convex   set over $V$ that contains $\mathbf{K}$.

For a subset $K$ of a Banach space $X$ and $1 \le p \le \infty$, it is not difficult to see that $K$ is relatively $p$-compact if and only if its absolutely convex hull is relatively $p$-compact.
We will now present an operator space version of this fact. Webster proved the case $p=\infty$ in \cite[p. 8]{webster1998matrix}.

\begin{lemma}\label{lemma-amconv-preserves-p-compactness}
Let $\mathbf{K}=(K_n)_n$ be a matrix set over $V$, and $1\le p \le \infty$. Then $\mathbf{K}$ is relatively operator $p$-compact if and only if so is $\amconv(\mathbf{K})$. Moreover, in this case $\frak m_p^o(\mathbf{K}) =\frak m_p^o(\amconv(\mathbf{K}))$.
\end{lemma}

\begin{proof}
Since $\mathbf{K} \subseteq \amconv(\mathbf{K})$, it is clear that if $\amconv(\mathbf{K})$ is relatively operator $p$-compact then so is $\mathbf{K}$ and moreover $\frak m_p^o(\mathbf{K}) \le \frak m_p^o(\amconv(\mathbf{K}))$.

Suppose now that $\mathbf{K}$ is operator $p$-compact. Let $v \in \mathcal{S}_p[V]$ such that $\mathbf{K}\subset \Theta^v(\mathbf{B}_{\mathcal{S}_p'})$. 
Since $\mathbf{B}_{\mathcal{S}_p'}$ is absolutely matrix convex and $\Theta^v$ is linear, $\Theta^v(\mathbf{B}_{\mathcal{S}_p'})$ is also absolutely matrix convex.
Therefore $\amconv(\mathbf{K})\subset \Theta^v(\mathbf{B}_{\mathcal{S}_p'})$, so $\amconv(\mathbf{K})$ is relatively operator $p$-compact, and $\kappa_p^o(\amconv(\mathbf{K})) \le \n{v}_{\mathcal{S}_p[V]}$. Taking the infimum over all such $v$ yields $\frak m_p^o(\amconv(\mathbf{K})) \le \frak m_p^o(\mathbf{K})$ as desired.
\end{proof}

The second part of the following result should be understood as the matrix version of the fact that an absolutely convex set is closed under taking linear combinations with coefficients in the unit ball of $\ell_1^n$.
Given a matrix $x \in M_n(V)$ we can also consider the canonical associated mapping $\Theta^x : \mathcal{S}_1^n \to V$. This notation is used in the following lemma.

\begin{lemma}\label{lemma:abs-matrix-convex-properties}
Let $V$ be an operator space and $\mathbf{K} = (K_n)$ an absolutely matrix convex  set over $V$.
\begin{enumerate}[(a)]
\item If $\rho : M_n \to M_m$ is a complete contraction, then $(\rho \otimes Id_V) K_n \subseteq K_m$.
\item If $\xi \in M_m(\mathcal{S}_1^n)$ with $\n{\xi} \le 1$ and $x \in K_n$, then
$(\Theta^x)_m(\xi) \in K_m$.
\end{enumerate}
\end{lemma}

\begin{proof}
$(a)$ If $\rho : M_n \to M_m$ is a complete contraction, by \cite[Cor. 5.3.5 (i)]{Effros-Ruan-book} there exist contractive matrices $\alpha \in M_{m,mn^2}$, and $\beta \in M_{mn^2,m}$ such that for any $y \in M_n$,
\begin{equation}\label{eqn-Stinespring}
\rho(y) = \alpha( \underbrace{y \oplus \cdots \oplus y}_{mn \text{ times}} )\beta.    
\end{equation}
 Thus, if $x \in K_n \subseteq M_n(V)$, it follows that $(\rho \otimes Id_V) x = \alpha( {x \oplus \cdots \oplus x} )\beta$, which belongs to $K_m$ by the absolute matrix convexity of $\mathbf{K}$.
 
 $(b)$ By the canonical duality $M_m(\mathcal{S}_1^n) \equiv \CB(M_n, M_m)$ and since $\n{\xi} \le 1$, note that $\xi$ induces a complete contraction ${\rho : M_n \to M_m}$ 
 which satisfies $(\rho\otimes Id_V)(x) = (\Theta^x)_m(\xi)$ for every $x \in M_n(V)$.
Indeed, by linearity it suffices to check this equality when $x$ has only one nonzero entry, say $x = E_{i_0j_0} \otimes v_0$ where $E_{ij} \in M_n$ are the matrix units and $v_0 \in V$.
Writing $\xi = ( (\xi^{kl}_{ij})_{i,j=1}^n )_{k,l=1}^m$, on one hand we have $(\Theta^{E_{i_0j_0} \otimes v_0})_m(\xi) = (\xi^{kl}_{i_0j_0}v_0)_{k,l=1}^m$.
On the other hand, $(\rho\otimes Id_V)(x) = (\rho\otimes Id_V)( E_{i_0j_0} \otimes v_0 ) = \rho(E_{i_0j_0}) v_0 = (\xi^{kl}_{i_0j_0})_{k,l=1}^mv_0$.
 The desired result now is deduced from $(a)$.
\end{proof}

\begin{remark}
As pointed out to us by an anonymous referee, it is worth mentioning that the aforementioned result \cite[Cor. 5.3.5]{Effros-Ruan-book} is a particular case of the well-known Stinespring dilation theorem.
Such a statement might seem puzzling if one is only familiar with said theorem in the case of completely positive maps, whereas in the Lemma above there is no such requirement on the map $\rho$.
However, there are versions of Stinespring's theorem without the complete positivity: for example, see \cite[Thm. 5.3.3]{Effros-Ruan-book}.
See also \cite[Cor. 5.3.5 (ii) and (iii)]{Effros-Ruan-book} and \cite[Sec. 2.2.2]{Watrous} for  results relating various properties of the map $\rho$ with different forms of the Stinespring dilation.
\end{remark}

The following is a version of the Hahn-Banach theorem for matrix convexity \cite[Thm. 2.3]{CD-Oikhberg} (essentially proved in  \cite[Prop. 4.1]{Effros-Webster}).

\begin{theorem}\label{thm-Hahn-Banach-matrix-convexity}
	Let $\mathbf{K} = (K_n)_n$ be a closed absolutely matrix convex   set over $V$ and let $v_0 \in M_n(V) \setminus K_n$ for some $n\in\N$.
	Then there exists $v' \in M_n(V')$ such that for all $m\in\N$ and all $v \in K_m$,
	$$
	\n{ \mpair{v'}{v} }_{M_{mn}} \le 1 \quad \text{but} \quad \n{ \mpair{v'}{v_0} }_{M_{n^2}} > 1,
	$$ 
	where
	$$
	\mpair{(v'_{ij})_{ij}}{(v_{kl})_{kl}} = \big( v'_{ij}(v_{kl}) \big)_{ijkl}.
	$$
\end{theorem}

It is clear that the closed unit ball of $\ell_1$ is the closure of the absolutely convex hull of the canonical basis. More generally, loosely speaking, the closed unit ball of an $\ell_1$-sum of Banach spaces $\ell_1( \{ X_i\}_{i \in I} )$ is the closure of the absolutely convex hull of the union of the individual balls $B_{X_i}$ (specifically, the union of the unit balls of the canonical copies of the $X_i$ inside $\ell_1( \{ X_i\}_{i \in I} )$).
We present an operator space version of this fact:

\begin{proposition}\label{prop-matrix-unit-ball-of-ell1-sum}
Let $\{V_i\}_{i \in I}$ be a collection of operator spaces. For each $i_0\in I$ let $J_{i_0} : V_{i_0} \to \ell_1( \{ V_i\}_{i \in I} )$ be the canonical complete isometry, that is, the one that sends $v \in V_{i_0}$ to the vector with $v$ in the $i_0$-th position and 0 everywhere else.
Then the closed matrix unit ball of $\ell_1( \{ V_i\}_{i \in I} )$ is the closure of the absolutely matrix convex hull of $\bigcup_{i\in I} J_i(\mathbf{B}_{V_i})$.
\end{proposition}

\begin{proof}
Since each $J_i$ is a complete isometry, it is clear that  $\bigcup_{i\in I} J_i(\mathbf{B}_{V_i}) \subseteq \mathbf{B}_{\ell_1( \{ V_i\}_{i \in I} )}$ and therefore
$\overline{\amconv\Big( \bigcup_{i\in I} J_i(\mathbf{B}_{V_i}) \Big)} \subseteq \mathbf{B}_{\ell_1( \{ V_i\}_{i \in I} )}$.
If the inclusion is strict, there is $n\in\N$ and $x_0 \in  B_{M_n( \ell_1( \{ V_i\}_{i \in I} ) )}$ which does not belong to $\overline{\amconv\Big( \bigcup_{i\in I} J_i(\mathbf{B}_{V_i}) \Big)}$, so by Theorem \ref{thm-Hahn-Banach-matrix-convexity} we can find $v' \in M_n(\ell_1( \{ V_i\}_{i \in I} )') \equiv \CB( \ell_1( \{ V_i\}_{i \in I}) ,M_n)$ such that $\n{ \mpair{v'}{x_0} }_{M_{n^2}} > 1$
but for all $m\in\N$ and all $x$ in the $m$-th level of $\amconv\Big( \bigcup_{i\in I} J_i(\mathbf{B}_{V_i}) \Big)$ we have $\n{ \mpair{v'}{x} }_{M_{mn}} \le 1$. 
The first condition says that $\n{v'}_{\CB(\ell_1( \{ V_i\}_{i \in I} ),M_n)}>1$, whereas the second says that for every $i\in I$ we have $\n{v'J_i}_{\CB(V_i,M_n)} \le 1$, which contradicts the universal property of $\ell_1$-sums of operator spaces.
\end{proof}

As announced, now we associate to a given operator $p$-compact matrix set a canonical operator $p$-compact mapping. For this, given $x$ an element of $\mathbf{K} = (K_n)_n$, a matrix set over the operator space $V$, we denote by $n_x \in \N$  the matrix level to which $x$ belongs, that is, $x \in M_{n_x}(V)$.
If $\mathbf{K}$ is completely bounded, define a map  $u_\mathbf{K} : \ell_1\big( \{ \mathcal{S}^{n_x}_1 \}_{x \in \mathbf{K}}\big) \to V$ by
\[
 \xi = \big( \xi(x)  \big)_{x\in \mathbf{K}} \mapsto \sum_{x \in \mathbf{K}} \Theta^x(\xi(x)).
\]
Note that this indeed makes sense, since $\n{\Theta^x(\xi(x))} \le \n{\Theta^x}_{\CB(\mathcal{S}_1^{n_x},V)} \n{\xi(x)}_{\mathcal{S}_1^{n_x}} = \n{x}_{M_n(V)} \n{\xi(x)}_{\mathcal{S}_1^{n_x}}$, so when $\mathbf{K}$ is completely bounded and $\xi \in \ell_1\big( \{ \mathcal{S}^{n_x}_1 \}_{x \in \mathbf{K}}\big)$ the series above is absolutely convergent.

The following theorem is a noncommutative partial version of  \cite[Prop. 3.5]{Delgado-Pineiro-Serrano-adjoints}.

\begin{theorem}\label{thm:equivalences-p-compact-u}
Let $\mathbf{K}$ be a completely bounded matrix set over $V$.
The following are equivalent:
\begin{enumerate}[(i)]
\item $\mathbf{K}$ is relatively operator $p$-compact.
\item $u_\mathbf{K} : \ell_1\big( \{ \mathcal{S}^{n_x}_1 \}_{x \in \mathbf{K}}\big) \to V$ is operator $p$-compact.
\end{enumerate}
Moreover, in this case $\frak m_p^o(\mathbf{K}) = \kappa_p^o(u_\mathbf{K})$.
\end{theorem}

\begin{proof}
For $y\in \mathbf{K}$, denote by $J_y : \mathcal{S}_1^{n_y} \to \ell_1( \{ \mathcal{S}^{n_x}_1 \}_{x \in \mathbf{K}}) $ the canonical inclusion. It is obvious from the definition of $u_{\mathbf{K}}$ that for each such $y$ we have $u_{\mathbf{K}}J_y = \Theta^y$.

$(ii) \Rightarrow (i)$:
Fix $y\in \mathbf{K}$.
Let $\xi \in M_{n_y}(\mathcal{S}_1^{n_y}) \equiv \CB(M_{n_y},M_{n_y})$ be the matrix of norm one that corresponds to the identity map $M_{n_y} \to M_{n_y}$, which is the matrix $\xi = (E_{kl})_{k,l=1}^{n_y}$ of matrix units.
Since $\Theta^yE_{kl} = y_{kl}$ for all $1 \le k,l \le n_y$, we have that $y = (u_{\mathbf{K}}J_y)_{n_y}(\xi)$.
Since $J_y$ is a complete isometry we know $J_y(\mathbf{B}_{\mathcal{S}_1^{n_y}}) \subseteq \mathbf{B}_{\ell_1( \{ \mathcal{S}^{n_x}_1 \}_{x \in \mathbf{K}}) }$, and therefore we have proved
$$
\mathbf{K} \subseteq u_{\mathbf{K}}\big( \mathbf{B}_{\ell_1( \{ \mathcal{S}^{n_x}_1 \}_{x \in \mathbf{K}}) } \big).
$$

Now, using that $u_{\mathbf{K}}$ is operator $p$-compact we conclude $\mathbf{K}$ is relatively operator $p$-compact and $\frak m_p^o(\mathbf{K}) \le \kappa_p^o(u_\mathbf{K})$.

$(i) \Rightarrow (ii)$:
For $x\in \mathbf{K}$, $m\in\N$ and $\xi \in B_{M_m(\mathcal{S}_1^{n_x})}$, note that $(u_{\mathbf{K}}J_x)_m(\xi) = (\Theta^x)_m(\xi)$
which belongs to $\amconv(\mathbf{K})$ by Lemma \ref{lemma:abs-matrix-convex-properties}.
This shows that $u_{\mathbf{K}} \Big( \bigcup_{x\in \mathbf{K}} J_x(\mathbf{B}_{\mathcal{S}_1^{n_x}}) \Big) \subseteq \amconv(\mathbf{K})$, from where it follows by linearity of $u_{\mathbf{K}}$ that
$u_{\mathbf{K}} \big( \amconv \Big( \bigcup_{x\in \mathbf{K}} J_x(\mathbf{B}_{\mathcal{S}_1^{n_x}}) \Big) \Big)\subseteq \amconv(\mathbf{K})$ and therefore, since $u_\mathbf{K}$ is continuous,
\[
u_{\mathbf{K}} \Big( \overline{\amconv \Big( \bigcup_{x\in \mathbf{K}} J_x(\mathbf{B}_{\mathcal{S}_1^{n_x}})} \Big) \Big) \subseteq
\overline{u_{\mathbf{K}} \Big( \amconv \Big( \bigcup_{x\in \mathbf{K}} J_x(\mathbf{B}_{\mathcal{S}_1^{n_x}}) \Big) \Big)}\subseteq \overline{\amconv(\mathbf{K})}.
\]
The desired conclusion now follows from Lemmas \ref{lemma-p-compact-closure} and 
 \ref{lemma-amconv-preserves-p-compactness},
and Proposition \ref{prop-matrix-unit-ball-of-ell1-sum}.
\end{proof}

\begin{remark} 
Note that the version of Theorem \ref{thm:equivalences-p-compact-u} for operator weakly $p$-compact matrix sets/mappings is also true, because Lemmas \ref{lemma-p-compact-closure} and 
 \ref{lemma-amconv-preserves-p-compactness} have versions for  operator weakly $p$-compact matrix sets (with analogous proofs).
More generally, if we have a class of matrix sets which is invariant under taking closures and absolutely matrix convex hulls, a version of Theorem \ref{thm:equivalences-p-compact-u} holds for the associated class of linear mappings. For example, a matrix set $\mathbf{K}$ is completely bounded if and only if so is $u_{\mathbf{K}}$.
\end{remark}

We will say that a matrix set $\mathbf{K} = (K_n)_n$ has \emph{finite height}
if there exists $N\in\N$ such that $K_n = \emptyset$ for $n \ge N$.
For such matrix sets and under the assumption of local reflexivity, we can get a full noncommutative version of \cite[Prop. 3.5]{Delgado-Pineiro-Serrano-adjoints}.

\begin{theorem}\label{thm:equivalences-p-compact-j}
Let $\mathbf{K}$ be a completely bounded matrix set of finite height over a locally reflexive operator space $V$.
The following are equivalent:
\begin{enumerate}[(i)]
\item $\mathbf{K}$ is relatively operator $p$-compact.
\item $u_\mathbf{K} : \ell_1\big( \{ \mathcal{S}^{n_x}_1 \}_{x \in \mathbf{K}}\big) \to V$ is operator $p$-compact.
\item $u_{\mathbf{K}}' : V' \to \ell_\infty\big( \{M_{n_x}\}_{x \in \mathbf{K}} \big)$ is completely $p$-nuclear.
\end{enumerate}
Moreover, in this case $\frak m_p^o(\mathbf{K}) = \kappa_p^o(u_\mathbf{K})=\nu_p^o(u_\mathbf{K}')$.
\end{theorem}

\begin{proof} 
$(i) \Leftrightarrow (ii)$: this was proved in
 Theorem \ref{thm:equivalences-p-compact-u}.

$(ii) \Rightarrow (iii)$:
By Proposition \ref{prop:adjoint-of-p-compact}, $u_\mathbf{K}'$ is quasi completely $p$-nuclear and $q\nu_p^o(u_\mathbf{K}') \le \kappa_p^o(u_\mathbf{K})$.
Since $\ell_\infty\big( \{M_{n_x}\}_{x \in \mathbf{K}} \big)$ is injective, it follows that $u_\mathbf{K}'$ is completely $p$-nuclear and $\nu_p^o(u_\mathbf{K}') \le \kappa_p^o(u_\mathbf{K})$.

$(iii) \Rightarrow (ii)$:
By dualizing the  commutative diagram associated to the completely $p$-nuclear mapping $u_\mathbf{K}'$, we obtain that $u_\mathbf{K}'' \in \mathcal{N}^p_o\big( (\ell_\infty\big( \{M_{n_x}\}_{x \in \mathbf{K}} \big))',V''\big)$
with $\nu^p_o(u_\mathbf{K}'') \le \nu_p^o(u_\mathbf{K}')$
(note that in the case $p=1$, we need to use Lemma \ref{lemma:multiplication-operator-on-B(ell_2)})
and therefore $\iota_V u_\mathbf{K} \in \mathcal{N}^p_o\big(  \ell_1\big( \{ \mathcal{S}^{n_x}_1 \}_{x \in \mathbf{K}}\big)  , V'' \big)$ with
$$
\nu^p_o(\iota_V u_{\mathbf{K}} : \ell_1\big( \{ \mathcal{S}^{n_x}_1 \}_{x \in \mathbf{K}}\big) \to V'' ) \le \nu_p^o(u_{\mathbf{K}}').
$$
Since $\mathbf{K}$ has finite height, by Lemma \ref{lemma-ell_infty-has-CMAP} $ \big( \ell_1\big( \{ \mathcal{S}^{n_x}_1 \}_{x \in \mathbf{K}}\big) \big)' = \ell_\infty\big( \{ M_{n_x} \}_{x \in \mathbf{K}}\big) $ has CMAP.
The space $\ell_1\big( \{ \mathcal{S}^{n_x}_1 \}_{x \in \mathbf{K}}\big)$ is strongly locally reflexive because its dual is the von Neumann algebra $\ell_\infty\big( \{ M_{n_x} \}_{x \in \mathbf{K}}\big)$ \cite[Thm. 15.3.5]{Effros-Ruan-book}.
Moreover $V$ is locally reflexive  and  $\ell_\infty\big( \{ M_{n_x} \}_{x \in \mathbf{K}}\big)$ too (if $\mathbf{K}$ has finite height $N$, this space is contained in  $\ell_\infty\big( \{ M_{N} \}_{x \in \mathbf{K}}\big)$, see the arguments for this case in the proof of Corollary \ref{cor-compactness-in-the-bidual-for-operators}), so by Theorem \ref{theorem-Pietsch} 
we get that $u_\mathbf{K} \in \mathcal{N}^p_o\big(  \ell_1\big( \{ \mathcal{S}^{n_x}_1 \}_{x \in \mathbf{K}}\big)  , V \big)$ with
$$
\nu^p_o(u_{\mathbf{K}} : \ell_1\big( \{ \mathcal{S}^{n_x}_1 \}_{x \in \mathbf{K}}\big) \to V ) \le \nu_p^o(u_{\mathbf{K}}').
$$
Since every completely right $p$-nuclear mapping is operator $p$-compact, we can now conclude that $u_\mathbf{K} \in \mathcal{K}_p^o\big(  \ell_1\big( \{ \mathcal{S}^{n_x}_1 \}_{x \in \mathbf{K}}\big)  , V \big)$ with $\kappa_p^o(u_{\mathbf{K}}) \le \nu_p^o(u_{\mathbf{K}}')$. 
\end{proof}

We remark that the map $u_{\mathbf{K}}'$ appearing in the previous result has a nice expression. Under the identification $ \big( \ell_1\big( \{ \mathcal{S}^{n_x}_1 \}_{x \in \mathbf{K}}\big) \big)' = \ell_\infty\big( \{ M_{n_x} \}_{x \in \mathbf{K}}\big) $, we have that
$$
u_\mathbf{K}' : V' \to \ell_\infty\big( \{M_{n_x}\}_{x \in \mathbf{K}} \big), \qquad v' \mapsto \big( \mpair{v'}{x}  \big)_{x\in \mathbf{K}}.
$$

Recall that in the classical case, a subset of a Banach space is relatively $p$-compact if and only if it is relatively $p$-compact in the bidual \cite[Thm. 2.4]{Galicer-Lassalle-Turco}.
We now prove an analogous result in the operator space setting, though with the extra hypotheses of local reflexivity and finite height.

\begin{proposition}\label{prop-compactness-in-bidual-for-sets}
Let $\mathbf{K}$ be a completely bounded matrix set of finite height over a locally reflexive operator space $V$.
The following are equivalent:
\begin{enumerate}[(i)]
\item $\mathbf{K}$ is relatively operator $p$-compact.
\item $\iota_V\mathbf{K}$ is relatively operator $p$-compact.
\item $\amconv(\mathbf{K})$ is relatively operator $p$-compact.
\end{enumerate}
Moreover,
$\frak m_p^o(\mathbf{K}) = \frak m_p^o(\iota_V \mathbf{K}) = \frak m_p^o(\amconv(\mathbf{K}))$.
\end{proposition}

\begin{proof}
$(i) \Leftrightarrow (iii)$: Follows from Lemma \ref{lemma-amconv-preserves-p-compactness}, including $\frak m_p^o(\mathbf{K}) = \frak m_p^o(\amconv(\mathbf{K}))$.

$(i) \Rightarrow (ii)$: It is obvious from the definition that if $\mathbf{K}$ is relatively operator $p$-compact then $\iota_V\mathbf{K}$ is also relatively operator $p$-compact, and $\frak m_p^o(\iota_V \mathbf{K}) \le \frak m_p^o(\mathbf{K})$.

$(ii) \Rightarrow (i)$: Suppose now that $\iota_V\mathbf{K}$ is relatively operator $p$-compact in $V''$. 
 By Theorem \ref{thm:equivalences-p-compact-j},
 $u_{\iota_V \mathbf{K}}' : V''' \to \ell_\infty\big( \{M_{n_x}\}_{x \in \mathbf{K}} \big)$ is completely $p$-nuclear (note that no additional hypotheses are needed for the implication $(i)\Rightarrow (iii)$ of the referred theorem).
 Taking a restriction, $u_\mathbf{K}' : V' \to \ell_\infty\big( \{M_{n_x}\}_{x \in \mathbf{K}} \big)$ is completely $p$-nuclear.
 Another application of Theorem \ref{thm:equivalences-p-compact-j} gives that $\mathbf{K}$ is relatively operator $p$-compact in $V$.
 Moreover, all of the steps above are quantitative and we get $\frak m_p^o(\mathbf{K}) \le \frak m_p^o(\iota_V \mathbf{K})$.
\end{proof}

Lemma \ref{lemma-N-maximal-amconv} below provides a characterization of $N$-maximal operator spaces which is well-known to specialists, but we have not been able to find a reference for it in the literature so we include its proof.

\begin{lemma}\label{lemma-N-maximal-amconv}
An operator space $V$ is $N$-maximal if and only if $\mathbf{B}_V = \overline{\amconv (\mathbf{K})}$, where $\mathbf{K} = (B_{M_n(V)})_{n=1}^N$.
\end{lemma}

\begin{proof}
Suppose first that $\mathbf{B}_V = \overline{\amconv(\mathbf{K})}$, and let $T : V \to W$ be a linear map. Note that $\n{T_N} \le C$ means precisely that $T(\mathbf{K}) \subseteq C \mathbf{B}_W$. Since
\begin{equation}\label{eqn-lemma-N-maximal-amconv}
T(\mathbf{B}_V) = T \Big( \overline{\amconv(\mathbf{K}})  \Big) \subseteq \overline{ T\big( \amconv(\mathbf{K})\big) } = \overline{ \amconv(T(\mathbf{K}))},
\end{equation}
we have that $T(\mathbf{B}_V) \subseteq C \mathbf{B}_W$, meaning that $\n{T}_{\cb} \le C$, which shows that $V$ is $N$-maximal. 

Suppose now that $V$ is $N$-maximal. Clearly $\overline{\amconv (\mathbf{K})} \subseteq \mathbf{B}_V$. If they were different, there exists some $v_0 \in B_{M_k(V)}$ which is not in $\overline{\amconv (\mathbf{K})} $, so in particular $k > N$.
By the Hahn-Banach theorem for matrix convexity (Theorem \ref{thm-Hahn-Banach-matrix-convexity}), we can find $v' \in M_k(V')$ such that for all $1 \le m\le N$ and all $v \in K_m$,
$\n{ \mpair{v'}{v} }_{M_{mk}} \le 1$, but $\n{ \mpair{v'}{v_0} }_{M_{k^2}} > 1$.
Under the identification $M_k(V') = \CB(V,M_k)$, $v'$ then corresponds to a mapping $T \in\CB(V,M_k)$ such that $\n{T_N} \le 1$ but $\n{T}_{\cb}>1$, contradicting the fact that $V$ is $N$-maximal.
\end{proof}

The previous lemma provides a crucial relationship between an $N$-maximal operator space and the matrix set of finite height given by the first $N$ levels of its closed matrix unit ball. This allows us to provide the following  more conceptual proof of Corollary \ref{cor-compactness-in-the-bidual-for-operators} (about regularity of the ideal of operator $p$-compact mappings) based on Proposition \ref{prop-compactness-in-bidual-for-sets}.

\begin{proof}[Alternative proof of Corollary \ref{cor-compactness-in-the-bidual-for-operators}]
The inequality $\kappa_p^o(\iota_WT) \le \kappa_p^o(T)$ follows from the ideal property.
Since $V$ is $N$-maximal, by Lemma \ref{lemma-N-maximal-amconv} we have that $\mathbf{B}_V$ is the closure of the absolutely matrix convex hull of a matrix set of finite height $\mathbf{K}$.
By assumption, $\iota_W T(\mathbf{B}_V)$ is relatively operator $p$-compact. Since $
\iota_W T(\mathbf{B}_V) \supseteq  \iota_W T(\mathbf{K})$, the latter is relatively operator $p$-compact as well.
By Proposition \ref{prop-compactness-in-bidual-for-sets}, $T(\mathbf{K})$ is relatively operator $p$-compact. By \eqref{eqn-lemma-N-maximal-amconv},
Lemmas \ref{lemma-p-compact-closure} and \ref{lemma-amconv-preserves-p-compactness} imply that $T(\mathbf{B}_V)$ is relatively operator $p$-compact, meaning that $T\in\mathcal{K}_p^o(V,W)$. Moreover, all the steps above are quantitative and one gets $ \kappa_p^o(T) \le \kappa_p^o(\iota_WT)$. 
\end{proof}

\begin{remark}
There are several results throughout the paper where our arguments needed some technical conditions, e.g. the complementation in the bidual in Proposition \ref{prop:adjoint-of-p-compact-finite-dimension}, or the coexactness in Theorem \ref{thm-adjoint-of-p-summing}.
However, we do not know if these conditions are necessary.
\end{remark}

\section*{Acknowledgements}

We thank T. Oikhberg for pointing out \cite[Example 3.3.1.3]{Junge-Habilitationschrift}.

\providecommand{\bysame}{\leavevmode\hbox to3em{\hrulefill}\thinspace}
\providecommand{\MR}{\relax\ifhmode\unskip\space\fi MR }
\providecommand{\MRhref}[2]{%
  \href{http://www.ams.org/mathscinet-getitem?mr=#1}{#2}
}
\providecommand{\href}[2]{#2}

\end{document}